\documentclass[a4paper,11pt]{article}

\usepackage[T1]{fontenc}
\usepackage[utf8]{inputenc}
\usepackage[bookmarks = true, hidelinks]{hyperref}
\usepackage{amsmath} 
\usepackage{mathtools} 
\usepackage{amsthm} 
\usepackage{amssymb} 
\usepackage{mathrsfs} 
\usepackage[capitalise]{cleveref} 
\usepackage[english]{babel} 
\usepackage{enumitem} 
\usepackage{lipsum} 
\usepackage{graphicx} 
\usepackage{tikz} 
\usetikzlibrary{calc, matrix} 
\usepackage{tikz-cd} 
\usepackage{xcolor} 
\usepackage{eucal} 
\DeclareMathAlphabet{\mathpzc}{OT1}{pzc}{m}{it} 
\usetikzlibrary{external}

\usepackage[a4paper,
left = 1in, 
right = 1in,
top = 1in,
bottom = 1.3in,
marginparwidth = 1.9in,
marginparsep = 0.2in,
headsep = 30pt]{geometry} 
\usepackage{parskip} 
\theoremstyle{definition}
\newtheorem{thm}{Theorem}
\newtheorem*{thm*}{Theorem}
\numberwithin{thm}{section}
\newtheorem{lem}[thm]{Lemma}
\newtheorem{prop}[thm]{Proposition}

\newtheorem{rmk}[thm]{Remark}
\numberwithin{figure}{section}
\numberwithin{table}{section}

\newcommand{\name}{\operatorname}
\newcommand{\bb}{\mathbb}
\newcommand{\mf}{\mathfrak}
\newcommand{\mc}{\mathcal}

\newcommand{\mb}{\mathbf}
\newcommand{\mi}{\mathsf}
\newcommand{\mr}{\mathrm}
\newcommand{\Hom}{\operatorname{Hom}}

\newcommand{\End}{\operatorname{End}}
\newcommand{\Aut}{\operatorname{Aut}}

\newcommand{\Tor}{\operatorname{Tor}}
\newcommand{\Ext}{\operatorname{Ext}}

\newcommand{\K}{\operatorname{Ker}}
\newcommand{\Ck}{\operatorname{Coker}}
\newcommand{\Id}{\operatorname{Id}}

\newcommand{\Ra}{\Rightarrow}

\newcommand{\xra}{\xrightarrow}

\newcommand{\hra}{\hookrightarrow}

\newcommand{\op}{\text{op}}

\newcommand{\wh}{\widehat}

\newcommand{\blank}{{\,\cdot\,}}

\newcommand{\T}{\mathbb{T}}

\newcommand{\StMod}{\mathsf{StMod}}
\newcommand{\Mod}{\mathsf{Mod}}

\newcommand{\Pic}{\textnormal{Pic}}
\newcommand{\pic}{\mathfrak{pic}}

\newcommand{\opgreek}[1]{\begingroup\mathgroup-1 #1\endgroup}
\DeclareMathOperator{\ho}{\opgreek{\pi}}
\newcommand{\thra}{\twoheadrightarrow}
\newcommand{\proj}{(\text{proj})}
\newcommand{\Ind}{\name{Ind}}

\author{Jeroen van der Meer\\
		Richard Wong}
\title{\vspace{-1cm}{\LARGE Endotrivial modules for cyclic $p$-groups and generalized quaternion groups via Galois descent}\vspace{3mm}\\
	{\Huge }}

\begin{document}

\maketitle
\begin{abstract}
	\noindent In this paper, we investigate the group of endotrivial modules for certain $p$-groups. Such groups were already been computed by Carlson--Th{\'e}venaz using the theory of support varieties; however, we provide novel homotopical proofs of their results for cyclic $p$-groups, the quaternion group of order~8, and for generalized quaternion groups using Galois descent and Picard spectral sequences, building on results of Mathew and Stojanoska. Our computations provide conceptual insights into the classical work of Carlson--Th{\'e}venaz.
\end{abstract}

\tableofcontents

\pagebreak

\section{Overview}
\label{sectionoverview}

Throughout this paper, let $G$ denote a finite group, and let $k$ be a field of characteristic~$p$, where~$p$ divides the order of $G$ (i.e. the characteristic is modular). In this setting, one can study the representation theory of $G$ over~$k$. As $p$ divides $|G|$, Maschke's theorem fails, which infamously implies that the structural phenomena of representation theory over modular characteristics are wildly different than the usual theory over other characteristics. Central to modular representation theory, then, is the study of the structural property of the category of $kG$-modules.

One particular instance of this is the problem of computing the group of \textbf{endotrivial modules}
\[T(G) \coloneqq \big\{M \in \mi{Mod}(kG) \ | \ \End_k(M) \cong k \oplus P \big\}\,\text{.}\]
That is, the $k[G]$-modules $M$ such that the endomorphism module decomposes as the direct sum of~$k$, the trivial $kG$-module, and a projective $kG$-module $P$. This forms a group under tensor product. The group of endotrivial modules was first studied by Dade for elementary abelian groups $G \cong (C_p)^n$ \cite{dade}, who regarded endotrivial modules as a stepping stone towards the study of the more general endopermutation modules. Endotrivial modules over $p$-groups were later classified by work of Carlson and Th{\'e}venaz using purely representation-theoretic techniques, such the theory of support varieties, cf. for instance \cite{Carlson2000-ux}. The classification for arbitrary finite groups is an active problem that has been studied by numerous people.

The group of endotrivial modules can be approached through homotopy theory --- something we will make profound use of in this paper. In \cref{sectionintroduction}, we realise the group of endotrivial modules as the Picard group of the stable module $\infty$-category. In fact, we obtain a Picard space, which admits a decomposition coming from a limit decomposition of the stable module $\infty$-category. This decomposition is then shown to be amenable to spectral sequence techniques.

In certain cases, the decomposition of the stable module $\infty$-category can be viewed through the lens of Galois theory. We take up this topic in \cref{sectiongaloisdescent} and we use a result of Rognes to give new proofs of the decomposition for cyclic $p$-groups and quaternion groups.

Finally, in \cref{sectionendotrivial} we evaluate the limit spectral sequences associated to the decomposition of the stable module category to explicitly compute the group of endotrivial modules for cyclic $p$-groups (\cref{outcomeC2n}, \cref{outcomeCpn}) and generalized quaternion groups (\cref{resultq8}, \cref{resultQ2n}). Although these groups have already been computed, the method given here is entirely new. In particular, our approach allows for a new interpretation of the fact that the group of endotrivial modules over $Q_8$ depends on the arithmetic structure of the base field; we shall see that it arises naturally from a certain nonlinear differential in the limit spectral sequence (see \Cref{rmk:q8}).  Furthermore, our approach to $T(Q_{2^n})$ is independent of the computation for $T(Q_8)$, whereas this was a crucial step in the classical approach (see \Cref{rmk:q2n}).

\section{Background}
\label{sectionintroduction}

\subsection{Endotrivial modules as a Picard group}

Let $G$ denote a finite group, and let $k$ be a field of characteristic~$p$, where~$p$ divides the order of~$G$ (i.e. the characteristic is modular). Central to modular representation theory is the study of the structural properties of the category of $kG$-modules. One particular instance of this is the problem of computing the group of \textbf{endotrivial modules}
\[T(G) \coloneqq \big\{M \in \mi{Mod}(kG) \ | \ \End_k(M) \cong k \oplus P \big\}\,\text{,}\]
where $P$ denotes a projective module. That is, we're looking at $kG$-modules $M$ such that the endomorphism module decomposes as the direct sum of~$k$, the trivial $kG$-module, and a projective $kG$-module $P$. Endotrivial modules form a group under tensor product. The group of endotrivial modules can be approached through homotopy theory, and the goal of this subsection is to illustrate how this can be done.

The failure of Maschke's theorem implies that not all $kG$-modules are projective. One can then additively localize the category of $kG$-modules $\Mod(kG)$ with respect to the maps that factor through projective modules. The resulting localization is called the \textbf{stable module category} $\StMod(kG)$. It carries the structure of a tensor-triangulated category.

Given a symmetric monoidal category $(\mc{C}, \otimes, \mb{1}_{\mc{C}})$, one can study the \textbf{Picard group} of invertible objects:
\[\Pic\big(\mc{C}\big) \coloneqq  \big\{M \in \mc{C} \ | \ \exists N \textnormal{ such that } M \otimes N \simeq \mb{1}_{\mc{C}} \big\}\]
In particular, if $\mc{C} = \StMod(kG)$, then we claim that $\Pic(\mc{C})$ coincides with the group of endotrivial modules:
\[T(G) \cong \Pic\big(\StMod(kG)\big)\,\text{.}\]
That this is the case, appears to implicitly known to experts, but we haven't been able to find a proof of this fact in the literature, so we digress for a moment to verify it.

\begin{lem}
    \label{projectiveequivalence}
    Two $kG$-modules $M$ and $N$ are equivalent in $\StMod(kG)$ if and only if there exist projective modules $P$ and $Q$ such that $M \oplus P \cong N \oplus Q$.
\end{lem}

\begin{proof}
    If $M$ and $N$ are projectively equivalent, then the natural maps $f \colon M \hra M \oplus P \xrightarrow{\sim} N \oplus Q \thra N$ and $g \colon N \hra N \oplus Q \xrightarrow{\sim} M \oplus P \thra M$ form the desired equivalence, in that $g \circ f - \Id$ factors through $Q$, and $f \circ g - \Id$ factors through $P$. 
    
    The converse is taken from \cite{1340906}. Suppose we have maps $f \colon M \to N$ and $g \colon N \to M$ with
    \[\begin{split}
        g\circ f - \Id &= M  \xra{\phi_1} P \xra{\phi_2} M \\
        f \circ g - \Id &= N \xra{\psi_1} Q \xra{\psi_2} N
    \end{split}\]
    Define $f' \colon M \to N \oplus P$ as $(f,\phi_1)$, and $g' \colon N \oplus P \to M$ as $g - \phi_2$. Then $g'\circ f' = \Id$, so that $M \oplus \K g' \cong N \oplus P$. We're done if we show that $\K g'$ is projective. This follows once we verify that $f' \circ g' - \Id$ factors through a projective: indeed once we know this, we may observe that the map becomes $-\Id$ when restricted to $\K g'$ but after this restriction it of course still passes through this projective, so $\K g'$ becomes a summand thereof. So let's verify the claim. Simply observe that $f' \circ g' - \Id$ factors as 
    \[\begin{tikzcd}[column sep = small, row sep = large]
    N \oplus P \ar[rd,"(\psi_1{,}p{,}\phi_1 \circ g - \phi_1 \circ \phi_2)"'] \ar[rr,"f' \circ g' - \Id"] & & N \oplus P \\
    & Q \oplus P \oplus P \ar[ur, "(\psi_2 - f \circ \phi_2{,} p_2 - p_1)"'] & \end{tikzcd}\]
    and $Q \oplus P \oplus P$ is projective.
\end{proof}

From this lemma, we learn that a module $M$ is $\otimes$-invertible if there exists a module~$N$ such that $M \otimes N \oplus Q \cong k \oplus P$ for some projective modules $P$ and $Q$. By applying the Krull--Schmidt theorem to this, we deduce two observations:
\begin{itemize}[noitemsep, topsep = 0pt]
    \item The $Q$ is not needed and we can simply write $M \otimes N \cong k \oplus P$ for some projective module~$P$; 
    \item $M$ and $N$ split up as $M_0 \oplus \proj$ and $N_0 \oplus \proj$ where `$\proj$' will henceforth be shorthand for `some projective module which doesn't deserve its own symbol'.
\end{itemize}

\begin{lem}
    The Picard group of $\StMod(kG)$ is isomorphic to $T(G)$.
\end{lem}

\begin{proof}
    Suppose first that $M$ is endotrivial. As $M$ is finitely generated, we have $\End_k(M) \cong M \otimes M^*$, and so $M$ is $\otimes$-invertible with inverse $M^*$. Conversely, suppose $M$ is a $kG$-module with $\otimes$-inverse $N$. By the discussion above we may write $M \otimes N \cong k \oplus P$, and we have $M \cong M_0 \oplus \proj$ where $M_0$ is indecomposable.
    
    Have a look at the commutative diagram\vspace{0mm}
    \[\begin{tikzcd}[column sep = huge, row sep = huge]
        M \otimes N\ar[r,"\sim"] \ar[d,"1\otimes f"'] & k \oplus P \ar[d,"\pi"] \\
        M \otimes M^* \ar[r,"\name{ev}"'] & k 
        \end{tikzcd}\]
    Here $\pi$ is the projection map, $\name{ev}$ the evaluation map on $M \otimes M^*$, and $f$ is the map $N \to M^*$ sending $n$ to $\phi \colon m \mapsto \pi(m \otimes n)$, where `$m \otimes n$' really refers to its isomorphic image in $k \oplus P$. As $\pi$ admits a section, so does $\name{ev}$, which means $k$ is a summand of $M \otimes M^*$. By tensoring with~$N$ we see that $N$ is a summand of $M \otimes N \otimes M^*$. Now write
    \[\begin{split}
        M \otimes N \otimes M^* &\cong (k \oplus P) \otimes \big(M_0^* \oplus \proj\big) \\
        &\cong M_0^* \oplus \proj \end{split}\]
    This tells us that $N$ is either projective or $M_0^*$ plus something projective. In the former case,~$M$ was trivial and there was nothing to prove anyway, and in the latter case, we've found that $M \otimes M^* \oplus \proj \cong k \oplus P$, and by the discussion above this lemma, this implies the desired result.
\end{proof}

The stable module category is in fact the homotopy category of a stable symmetric monoidal $\infty$-category, which can be seen from the fact that $\StMod(kG)$ can be described as the Verdier quotient of the bounded derived category of $kG$-modules by the perfect complexes --- a result first proved in \cite{rickard1989derived}. This observation is what makes the study of endotrivial modules amenable to homotopical techniques. We remark that such an approach is not entirely new; similar themes can be found in \cite{grodal2016endotrivial} and \cite{barthel77torsion}. 

To any symmetric monoidal $\infty$-category we can in fact associate a Picard \textit{space} $\mc{P}\mr{ic}(\mc{C})$, defined as the $\infty$-groupoid underlying the full subcategory on the $\otimes$-invertible objects in $\mc{C}$.  This is an enhancement of the classical Picard group.

\begin{lem}
    \label{homotopygroupsofPic}
    The homotopy groups of the Picard space are as follows:
    \[\pi_t \mc{P}\mr{ic}(\mc{C}) \cong 
	\begin{cases}
	\Pic(\mc{C}) \qquad &\text{if $t = 0$;} \\
		\pi_0 (\Omega \mc{C})^\times \qquad &\text{if $t = 1$;} \\
		\pi_{t - 1}(\Omega \mc{C}) \qquad &\text{if $t \geq 2$.}\end{cases}\]
	Here $\Omega \mc{C}$ is shorthand for the $\bb{E}_\infty$-ring $\End(\mb{1}_{\mc{C}})$ of endomorphisms of the $\otimes$-unit.
\end{lem}

\begin{proof}[Proof sketch]
    Tensoring with an $\otimes$-invertible object tautologically describes an automorphism of $\mc{P}\mr{ic}(\mc{C})$. From this we observe that the Picard space decomposes as $\Pic(\mc{C}) \times B \Aut(\mb{1}_{\mc{C}})$. 
\end{proof}

\begin{rmk}
\label{passingtoInd}
    In the literature the stable module $\infty$-category is often defined as the $\Ind$-completion of the aforementioned Verdier quotient. The passage to $\Ind$-completion will not alter any of our results --- in particular, the Picard space does not change. We take a moment to verify this.
\end{rmk}

\begin{proof}
    First recall that, for any (small) stable $\infty$-category $\mc{C}$, the $\Ind$-completion may be modeled by the functor category $\mi{Fun}_{\mr{ex}}(\mc{C}^\op,\mi{Sp})$, the natural map $\mc{C} \to \Ind(\mc{C})$ being the Yoneda embedding. If $X$ is an object in $\mc{C}$ then its image is compact in $\Ind(\mc{C})$, effectively for formal reasons, and the map $\mc{C} \to \Ind(\mc{C})^\omega$ identifies with the idempotent completion of $\mc{C}$ \cite[Lem.~5.4.2.4]{lurie2009higher}.
    
    If $\mc{C}$ is symmetric monoidal then the tensor product of $\Ind(\mc{C})$ is given by the Day convolution on the functor category, and the Yoneda embedding $\mc{C} \to \Ind(\mc{C})$ is a symmetric monoidal functor. The $\otimes$-unit $\mb{1}_{\mc{C}}$ gets mapped to a compact object in $\Ind(\mc{C})$; consequently, any $\otimes$-invertible object in $\Ind(\mc{C})$ is compact as well, and so any $\otimes$-invertible object in $\Ind(\mc{C})$ can in fact be found in $\Ind(\mc{C})^\omega$, the idempotent completion of $\mc{C}$. But $\mc{C} = \mi{Mod}_G^\omega(k) / \mi{Perf}(kG)$ is already idempotent complete, so no new $\otimes$-invertible objects can arise upon $\name{Ind}$-completing.
\end{proof}

\subsection{The limit spectral sequence}

Associating a Picard space to a stable symmetric monoidal $\infty$-category i functorial under exact symmetric monoidal functors, and so we have a functor $\mc{P}\mr{ic} \colon \mathsf{Cat}^\otimes \to \mc{S}_*$. This functor commutes with limits, cf. \cite[Proposition~2.2.3]{ms}, which means that, whenever we have a limit decomposition of $\infty$-categories, we have a corresponding limit decomposition of Picard spaces.

In view of this, it is natural to recall the following. Whenever we have a diagram $\mc{F} \colon \mathcal{I}^\op \to \mc{S}_*$ of pointed spaces, there is a spectral sequence 
\[ E_2^{st} = H^s(\mc{I},\pi_t \mc{F}) \Ra \pi_{t - s}\big( {\varprojlim}_{\mc{I}} \mc{F}\big)\]
whose $E_2$-page is given by the cohomology of the $\mi{Ab}$-valued presheaf $\pi_t\mc{F}$ over the diagram $\mc{I}$. The spectral sequence dates back to the work of Bousfield--Kan, cf. \cite{bousfield1972homotopy}. Unfortunately, the Bousfield--Kan spectral sequence suffers from convergence issues and admits fringe effects, which makes it unreliable from a computational perspective.

We may circumvent this convergence issue by passing to spectra. More precisely, whenever we have a diagram $\mc{F} \colon \mc{I}^\op \to \mi{Sp}$ of \textit{spectra}, there is a completely analogous spectral sequence (see \cite[Section~1.2.2]{lurie2012higher}) but which does not exhibit fringe effects.

Now, if $\mc{C}$ is a symmetric monoidal stable $\infty$-category, then the Picard space is a grouplike $\bb{E}_\infty$-space, and can thus be viewed as a connective spectrum, which we call the \textbf{Picard spectrum} of $\mc{C}$, denoted $\mf{pic}(\mc{C})$. The functor $\mf{pic} \colon \mi{Cat}_{\mr{st}}^\otimes \to \mi{Sp}_{\geq 0}$ commutes with limits as well, and so any limit decomposition of categories yields a limit decomposition of connective Picard spectra. 

It is worth pointing out that the limit spectral sequence lives in \textit{nonconnective} spectra, whereas the Picard spectra are \textit{connective} spectra. This is more than just a superficial difference, since the inclusion functor $\mi{Sp}_{\geq 0} \to \mi{Sp}$ does not commute with limits. However, a limit of connective spectra can be computed by taking a limit in the category of nonconnective spectra and then passing to the connective cover. Consequently, the discrepancy between a limit of connective spectra taken in $\mi{Sp}_{\geq 0}$ and taken in $\mi{Sp}$ is concentrated in negative degrees. In view of our main goal, which is to compute $\pi_0$ of Picard spectra, this discrepancy will never pose issues. 

Let's summarise our findings into a lemma.

\begin{lem}
    \label{picardspectralsequence}
    Let $\mc{C}$ be a symmetric monoidal stable $\infty$-category, and suppose that $\mc{C}$ is realised as the limit over a diagram $\mc{F} \colon \mc{I}^\op \to \mi{Cat}^\otimes_{\mr{st}}$. Then there is a spectral sequence
    \[ H^s(\mc{I},\pi_t \,\pic \,\mc{F}(\blank)\big) \Ra \pi_{t - s}\, {\varprojlim}_{\mc{I}}\, \pic(\mc{C}_i)\,\text{,}\]
    and for $* \geq 0$, $\pi_* \varprojlim_{\mc{I}} \pic(\mc{C}_i)$ may be identified with $\pi_* \pic(\mc{C})$. 
\end{lem}

\begin{rmk}
    The Picard spectrum admits a delooping by the Brauer spectrum. Brauer spectra ought to admit descent as well --- cf. for instance \cite[Prop.~3.45]{mathew} --- and so the $(-1)$-line for our limit spectral sequence of Picard spectra reveals information about the Brauer groups as well. We will find that the $(-1)$-line is often significantly more complicated than the nonnegative lines.
\end{rmk}

\subsection{Quillen stratification}

In view of \cref{picardspectralsequence}, it would be interesting to exhibit a limit decomposition for stable module categories, which is what we turn to in this section. 

Let $G$ be a finite group, and let $\mc{A}$ be a collection of subgroups of $G$ satisfying the following properties.
\begin{itemize}[noitemsep, topsep = 0pt]
	\item $\mc{A}$ is closed under finite intersections;
	\item $\mc{A}$ is closed under conjugation by elements of $G$;
	\item every elementary abelian $p$-subgroup of $G$ is contained in a member of $\mc{A}$.
\end{itemize}
For any such collection $\mc{A}$, we define $\mc{O}_{\mc{A}}(G)$ to be the full subcategory of the orbit category $\mc{O}(G)$ spanned by those objects $G / H$ for which $H$ is in $\mc{A}$. We have the following result, which can be found in \cite[Corollary 9.16]{mathew}, and which is effectively a higher-categorical elaboration of Quillen's stratification theory. 

\begin{thm}
	\label{quillen}
	If $G$, $k$ and $\mc{A}$ are as above, then the stable module category of $G$ decomposes as
	\[ \mi{StMod}(kG) \simeq \varprojlim_{G / H \in \mc{O}_{A}(G)^\op} \mi{StMod}(kH)\,\text{.}\]
\end{thm}
The functoriality is as follows. One can identify $\mi{StMod}(kH)$ as the category of modules objects over the commutative algebra object $k(G/H)$ in $\mi{StMod}(kG)$, as was proved by Balmer \cite{balmer} (cf. \cite{MNNnilpotence2015-jf} and \cref{rmk:abstractdescent}) --- that is,
\[\mathsf{StMod}(kH) \simeq \Mod_{\mathsf{StMod}(kG)}\big(k(G/H)\big)\,\text{,}
\]
and there is an obvious functor $\mc{O}_{\mc{A}}(G) \to \mi{CAlg}\big(\mi{StMod}(kG)\big)$ given by $G/H \mapsto \mi{StMod}(kH)$. 
This limit decomposition arises from restriction and coinduction. 


As discussed in the previous subsection, any limit decomposition of stable symmetric monoidal $\infty$-categories yields a corresponding decomposition of Picard spectra
\begin{equation} 
	\label{picdecomposition}
	\mf{pic} \,\mi{StMod}(kG) \simeq \varprojlim_{G / H \in \mc{O}_{A}(G)^\op} \mf{pic} \, \mi{StMod}(kH)\,\text{,}
\end{equation}
and hence a corresponding limit spectral sequence
\begin{equation}
    \label{picspectralsequence}
    E_2^{st} = H^s\big(\mc{O}_{\mc{A}}(G),\pi_t \mf{pic}\StMod(kH)\big) \Ra \pi_{t - s}\big(\mf{pic}\,{\StMod}(kG)\big)\,\text{.}
\end{equation}
To evaluate this spectral sequence, it is necessary to study both the objects and the differentials of this spectral sequence, which we shall do for the remainder of this section.

To understand the behaviour of the differentials in the spectral sequence, we make use of the computational tools developed in \cite[Part II]{ms}. We summarise their results here. We start with a symmetric monoidal stable $\infty$-category $\mc{C}$.

\begin{lem}[{\cite[Corollary 5.2.3]{ms}}]
	\label{functorialequivalence}
	One has a functorial equivalence $\tau_{[t, 2t-1]} \Aut(\mb{1}_{\mc{C}}) \simeq \tau_{[t, 2t-1]} \mb{1}_{\mc{C}}$, where $\tau_{[\blank,\blank]}$ denotes the truncation functor with homotopy groups in the specified range. This induces a functorial equivalence $\tau_{[t+1,2t]} \mf{pic}\,\mc{C} \simeq \Sigma \tau_{[t,2t-1]} \Omega \mc{C}$.
\end{lem}

Now, $\Omega$, too, commutes with limits, yielding a decomposition of $\Omega \mi{StMod}(kG)$ analogous to \cref{picdecomposition}, and hence a limit spectral sequence analogous to \cref{picspectralsequence}. \cref{functorialequivalence} then allows one to import differentials from the limit spectral sequence for $\Omega$ into the limit spectral sequence for Picard spectra. 

\begin{thm}[{\cite[Comparison Tool 5.2.4]{ms}}]
	\label{comparisontool}
	Let $\mc{F} \colon \mc{I} \to \mi{Cat}_{\mr{st}}^\otimes$ be a diagram of symmetric monoidal stable $\infty$-categories. Consider the limit spectral sequences
	\begin{align*}
	E_2^{st}(\mf{pic}) = H^s\big(\mc{I}, \pi_t \mf{pic}\,\mc{F}(\blank)\big) &\Ra \pi_{t - s} \mf{pic}\big({\varprojlim}_{\mc{I}} \mc{F}\big) \,\text{,} \\ 
	E_2^{st}(\Omega) = H^s\big(\mc{I}, \pi_t \Omega\,\mc{F}(\blank)\big) &\Ra  \pi_{t - s} \Omega\big({\varprojlim}_{\mc{I}} \mc{F}\big)\,\text{.}\end{align*}
	Then we have an equality of differentials $d_r^{st}(\mf{pic}) = d_r^{s,t-1}(\Omega)$ for all $(s,t)$ such that either $t - s > 0$ or $t \geq r + 1$.
\end{thm}

As it turns out, the spectral sequence for $\Omega$ is easier to understand, because the endomorphism spectra are $\bb{E}_\infty$-rings, which imbue the limit spectral sequence with a multiplicative structure. We will make frequent use of this advantage throughout our computations.

Let us now take a closer look at the endomorphism spectrum $\Omega \StMod(kG)$ of the unit. This spectrum can be described explicitly, but before we are able to give the description, we need to recall some relevant definitions.

If $X$ is a spectrum admitting a $G$-action, then we can capture the $G$-action as a functor $BG \to X$. We then associate to $X$ its homotopy orbits $X_{hG}$ and homotopy fixed points $X^{hG}$, defined as $\varinjlim_{BG} X$ and $\varprojlim_{BG} X$, respectively. There is a norm map $X_{hG} \to X^{hG}$ whose cofibre is called the Tate construction, denoted $X^{tG}$.

As we've seen, any limit of spectra has an associated limit spectral sequence. Applied to $X^{hG}$, we obtain what is commonly called the homotopy fixed point spectral sequence (HFPSS). A dual spectral sequence, called the homotopy orbit spectral sequence, exists for $X_{hG}$, as does a four-quadrant spectral sequence for $X^{tG}$, called the Tate spectral sequence, which we'll encounter in \cref{faithfulgaloisextensions}.

If $X$ is the Eilenberg--MacLane spectrum of a $G$-module $M$, then the homotopy groups of $M^{tG}$ carry classical arithmetic information. Specifically, the homotopy groups of $M_{hG}$ are given by
\[ \pi_t (M_{hG}) = \begin{cases}
    H_{t}(G;M) \qquad &\text{if $t \geq 0$,} \\
    0 \qquad &\text{otherwise,} \end{cases}\]
and the homotopy groups of $M^{hG}$  are given by
\[ \pi_{-t} (M^{hG}) = \begin{cases}
    H^{t}(G;M) \qquad &\text{if $t \geq 0$,} \\
    0 \qquad &\text{otherwise.} \end{cases}\]
Notice that $\pi_0(M_{hG})$ is the classical set $M_G$ of $G$-orbits while $\pi_0(M^{hG})$ is the set $M^G$ of fixed points. The norm map $M_{hG} \to M^{hG}$ is necessarily zero on nonzero homotopy groups, while the map on $\pi_0$ is the classical norm map $N \colon M_{hG} \to M^{hG}$ sending an orbit $\{gm\}$ to the sum $\sum_g gm$. Through the long exact sequence of homotopy groups associated to the fibre sequence $M_{hG} \to M^{hG} \to M^{tG}$, one then infers that
\[ \pi_{t}(M^{tG}) \cong \wh{H}^{-t}(G;M)\,\text{,}\]
where $\wh{H}$ denotes Tate cohomology, which is classically defined as
\[\wh{H}^{*}(G;M) = \begin{cases}
		H^*(G;M) \qquad &\text{if $* \geq 1$;} \\
		\Ck N \qquad &\text{if $* = 0$;} \\
		\K N \qquad &\text{if $* = -1$;} \\
		H_{-* - 1}(G;M) \qquad &\text{if $* \geq -2$.}
\end{cases}\]
If $X$ is an $\bb{E}_\infty$-ring, then so is $X^{tG}$. In particular, if $R$ is a classical ring with a $G$-action, then $\pi_* R^{tG}$ admits a cup product, which coincides with the ring structure on Tate cohomology. We refer the reader to \cref{appendixtate} for some explicit computations of Tate cohomology rings.

\begin{lem}
\label{omegaistate}
    There exists an equivalence of $\bb{E}_\infty$-rings
    \[ \Omega \StMod(kG) \simeq k^{tG}\,\text{,}\]
    where the $G$-action on $k$ is taken to be the trivial one.
\end{lem}

On the level of homotopy groups, this is reflected by the classical fact that, in the triangulated stable module category, 
\[\text{\underline{Hom}}_{kG}(\Omega^t k,k) \cong \wh{H}^{t}(G;k)\,\text{.}\]

\begin{rmk}
    If $G$ is in fact a $p$-group, then $\mi{StMod}(kG)$ is in fact equivalent to $\mi{Mod}(k^{tG})$. We digress for a while to verify that this is the case.
\end{rmk}

\begin{proof}
By the Schwede--Shipley theorem (\cite[Theorem~7.1.2.1]{lurie2012higher}), it suffices to show that the category has a compact generator, which we claim is $k$. Let us first see why $k$ is nontrivial object in $\mi{Mod}_G^\omega(k)/\mi{Perf}(kG)$ to begin with. This is true so long as $k$ isn't a perfect complex. Indeed this is the case: $k$ is not compact on $\mi{Mod}^\omega_G(k)$, effectively because $H^{-*}(G;k) = \pi_* \name{Map}_{\mi{Mod}^\omega_G(k)}(k,k)$ fails to be bounded in modular characteristic.

In general, if $\mc{C}$ is an $\infty$-category, then any object in $\mc{C}$ is compact in $\name{Ind}(\mc{C})$ for formal reasons, so $k$ defines a compact object. So then why does $k$ generate $\mi{StMod}\big(kG\big)$? By definition, one must verify that, for any object~$M$, if $\Hom(k,\Sigma^i M) = 0$ for all $i \in \bb{Z}$, then $M \simeq 0$. We may represent the $\Sigma^i M$ as $kG$-representations, so that the existence of a nonzero map $k \to \Sigma^i M$ may be identified with the existence of a fixed vector in the $kG$-representation $\Sigma^i M$. However, such nonzero maps always exist because $kG$-representations have a nonzero fixed vector. To see this, reduce to the case $k = \bb{F}_p$ by inspecting the underlying $\bb{F}_p$-vector space, and then apply a counting argument. This concludes the claim.
\end{proof}

Let us go back to \cref{picspectralsequence} again. In all examples of interest, we will take $\mc{A}$ to be a family of elementary abelian subgroups. In view of \cref{homotopygroupsofPic} and \cref{omegaistate}, the higher homotopy groups of $\pic \,\StMod(kH)$ are well-understood: they are given by Tate cohomology groups of elementary abeian groups. But what about the $0$-th homotopy groups of $\mf{pic}\big({\StMod}(kH)\big)$, i.e. the Picard groups? The endotrivial modules of elementary abelian groups are understood via a result by a theorem of Dade, which states that the Picard group is necessarily generated by the suspension of the unit. We will take the computation of the Picard group for elementary abelians as a starting point, working our way up from there. Let's capture it as a lemma.

\begin{lem}[\cite{dade}]
	\label{reference}
	The Picard group of the stable module category of elementary abelian groups is described as follows:
	\[ \pi_0 \,\mf{pic} \, \mi{StMod}\big(kC_p^n\big) \cong \begin{cases}
		0\qquad &\text{if $p = 2$ and $n = 1$;} \\
		C_2 \qquad &\text{if $p$ is odd and $n = 1$;} \\
		\bb{Z} \qquad &\text{if $n \geq 2$.}\end{cases}\]
\end{lem}

Let us revisit what we know so far. We have our spectral sequence \cref{picspectralsequence}, and we can compare this spectral sequence functorially with an analogous spectral sequence for $\Omega \StMod(kG)$. The latter has a multiplicative structure, and the $E_2$-page can be described in terms of Tate cohomology groups. In fact, the spectral sequence may be recognised as a rather classical one. To see this, let's suppose we may take $G$ to be a finite $p$-group with a single normal elementary abelian $p$-subgroup $H$. Then $\mc{O}_{\mc{A}}(G) \simeq B(G/H)$, and the limit spectral sequence for $\Omega \StMod(kG)$ reads
\[ E_2^{st} = H^s\big(G / H; \wh{H}^{-t}(H;k)\big) \Ra \wh{H}^{s - t}(G;k)\,\text{.}\]
For nonpositive $t$, the spectral sequence is indeed isomorphic to the Hochschild--Serre spectral sequence associated to the extension $H \to G \to G / H$. This spectral sequence is sufficiently well-studied that the differentials are known in all examples of interest. Via the Tate duality pairing (cf. \cref{tatereference}) this allows us to deduce the differentials for positive $t$ as well.

From a homotopical viewpoint, the comparison with the Hochschild--Serre spectral sequence can be seen by considering the natural map $k^{hH} \to k^{tH}$. The map is $G / H$-equivariant and induces an isomorphism on nonnegative homotopy groups, so that their limit spectral sequence may be compared. Applied to $k^{hH}$, this is the homotopy fixed points spectral sequence, and it converges to the homotopy groups of $(k^{hH})^{hG / H}$, which is naturally isomorphic to $k^{hG}$. 

Although a large swathe of differentials can now be understood using \cref{comparisontool} and the comparison with the Hochschild--Serre spectral sequence, we will often find that there's a particular differential which strongly influences the development of the $0$-line but which just barely falls outside the range of \cref{comparisontool}. For these differentials, we use an elegant formula of Mathew--Stojanoska. To match their statement with ours, let's assume that the diagram~$\mc{I}$ consists of a single object so that the limit spectral sequences becomes an HFPSS.

\begin{thm}[{\cite[Theorem 6.1.1]{ms}}]
	\label{unstabledifferential}
	Let the notation be as in \cref{comparisontool}. Assume that~$\mc{I}$ has a single object so that we may identify the limit spectral sequences with homotopy fixed point spectral sequences. Then we have the formula
	\[ d_r^{rr}(\mf{pic})(x) = d_r^{r,r-1}(\Omega)(x) + x^2\,\text{,}\]
	where the square refers to the multiplicative structure in the limit spectral sequence for $\Omega$.
\end{thm}

\begin{rmk}
    The classes of $p$-groups that are considered in this paper (cyclic $p$-groups and generalized quaternion groups) have a single elementary abelian $p$-subgroup. As a result, we can identify the limit spectral sequence with the homotopy fixed point spectral sequence and use \cref{unstabledifferential}. However, analogous methods apply to more complicated groups as well. This is currently work in progress.
\end{rmk} 

\section{Galois descent}
\label{sectiongaloisdescent}

\subsection{Stratification as Galois descent}

Mathew's \cref{quillen} becomes especially simple when we may take the family $\mc{A}$ to consist of a single (necessarily normal) subgroup $H$. In such examples, the decomposition reduces to the much simpler
\begin{equation}
\label{simpledecomposition}
    \mi{StMod}(kG) \simeq \mi{StMod}(kH)^{hG/H}\,\text{.}
\end{equation}
Throughout this paper, we will consider the two families of $p$-groups where this phenomenon occurs:
\begin{itemize}[noitemsep, topsep = 0pt]
    \item The cyclic $p$-groups $C_{p^n}$, and
    \item the generalized quaternion groups $Q_{2^n}$.
\end{itemize}
The cyclic $p$-groups obviously have a single elementary abelian subgroup $H$, which is isomorphic to~$C_p$. As for the generalized quaternion groups, we recall that these may defined e.g. algebraically as the groups
\begin{align*}
    Q_8 &\cong \langle \theta, \tau \ | \ \theta^4 = 1, \theta^2 = \tau^2, \tau\theta\tau^{-1} = \theta^{-1}\rangle\,\text{.} \\
    Q_{2^n} &\cong \langle \theta, \tau \ | \ \theta^{2^{n-1}} =\tau^4 = 1, \theta^{2^{n-2}} = \tau^2, \tau\theta\tau^{-1} = \theta^{-1}\rangle\,\text{.}
\end{align*}
The centre $H = Z(Q_{2^n})$ is the only nontrivial elementary abelian subgroup, being isomorphic to~$C_2$, and the quotient is isomorphic to the Klein four-group $(C_2)^2$ or the dihedral group $D_{2^{n-1}}$.

As it happens, we may re-interpret the decomposition as an instance of what is known as faithful Galois descent. We begin with the relevant definitions. Let $f \colon R \to S$ be a map of $\bb{E}_\infty$-ring spectra. Then we call $f$ a \textbf{$G$-Galois extension} if there is a $G$-action on $S$ such that the natural maps $R \to S^{hG}$ and $S \otimes_R S \to F(G_+,S)$ are weak equivalences. We say that the Galois extension is \textbf{faithful} if $S$ is moreover faithful as an $R$-module.

Whenever we have a faithful Galois extension, we have a good theory of descent called \textbf{Galois descent}, which has been studied by Gepner--Lawson \cite{GepnerLawson2016-ps} and Mathew--Stojanoska \cite{ms}, among others:

\begin{thm}
\label{thm:galoisdescent}
    If $f \colon R \to S$ is a faithful $G$-Galois extension of $\bb{E}_\infty$-rings, then we have a natural equivalence of $\infty$-categories
    \[ \mi{Mod}(R) \simeq \mi{Mod}(S)^{hG} \,\text{.} \]
\end{thm}

In view of \cref{omegaistate} and \cref{simpledecomposition}, it is clear that we expect a certain faithful $G/H$-Galois extension $k^{tG} \to k^{tH}$ whenever $G$ is any of the aforementioned groups. In the next subsection, we go about to prove this assertion, thus yielding a different proof of \cref{quillen} for these specific groups. This proof, we will find, is purely homotopical.

\begin{rmk}
    \cref{reference} is also known to have a proof using Galois descent, or rather, reverse Galois descent; see \cite{Mathewtorus2015-yg}. Briefly, if $A$ is an abelian $p$-group of $p$-rank $n$, then one can construct a fiber sequence of classifying spaces
    \[ B\bb{Z}^n \to BA \to B^2 \bb{Z}^n \simeq B \bb{T}^n \,\text{,}\]
    where $\bb{T}^n$ denotes the $n$-torus, which Mathew uses to prove that there exist faithful $\mathbb{T}^n$-Galois extensions of ring spectra 
    \[k^{h\T^n} \to k^{hA} \qquad \text{and} \qquad k^{t\T^n} \to k^{tA}\,\text{.}\]
    In this case, it is the source rather than the target which is understood well, and Mathew proves conditions for an element in the Picard group of $k^{t\mathbb{T}^n}$ to descend to the Picard group of $k^{tA}$. This, along with a computation of $\Pic(k^{t\mathbb{T}^n})$, shows that $\Pic(k^{tA})$ is cyclic.
\end{rmk}

\subsection{Galois extensions for cochain algebras}

In this section, we prove that for a $p$-group $G$ of $p$-rank $1$ and with unique nontrivial elementary abelian subgroup~$H$, we have natural maps $k^{tG} \to k^{tH}$ that are $G/H$-Galois extensions. We first show that we have Galois extension $k^{hG} \to k^{hH}$ on homotopy fixed points, and then invoke a base change argument to deduce the desired result. Our main tool is the following result of Rognes.

\begin{thm}[{\cite[Prop. 5.6.3]{rognes}}]
    \label{rognesgalois}
    Let $\Gamma$ be a finite discrete group, and $P \to X$ a principal $\Gamma$-bundle.  Suppose that $X$ is path-connected and $\pi_1(X)$ acts nilpotently on $H_*(\Gamma;k)$. Then the map of cochain $k$-algebras $F(X_+,k) \to F(P_+,k)$ is a $\Gamma$-Galois extension.
\end{thm}

We sketch the idea of the proof. To see that $F(X_+,k) \simeq F(P_+,k)^{h\Gamma} $ follows from properties of principal $\Gamma$-bundles. Namely, we have that the (right) action of $\Gamma$ on $P$ induces a left $\Gamma$ action on $F(P_+,k)$, which is compatible with the identification $X \simeq P_{h\Gamma}$. That is, of $X$ as the homotopy orbits of the action of $\Gamma$ on $P$.

The interesting part is showing that 
\[F(P_+,k) \otimes_{F(X_+,k)} F(P_+,k) \simeq F\big(\Gamma_+, F(P_+,k)\big)\,\text{.}\] 
As a tensor product of ring spectra, the homotopy groups of the left-hand side may be computed using the K{\"u}nneth spectral sequence
\[E^2_{s,t} \cong \Tor_{s,t}^{\pi_*(A)}(\pi_*(B), \pi_*(B)) \Rightarrow \pi_{s+t}(B \otimes_A B)\,\text{.}\]
Meanwhile, via the identification $F\big(\Gamma_+, F(P_+,k)\big) \simeq F\big((\Gamma \times P)_+,k\big)$, the homotopy groups of the right-hand side can be computed using the Eilenberg--Moore spectral sequence with $k$-coefficients, 
\[E^2_{s,t} \cong \Tor_{s,t}^{H^*(X;k)}\big(H^*(P;k), H^*(P;k)\big) \Rightarrow H^{-(s+t)}(P \times_X P;k)\,\text{.}\]

Since $\pi_*(F(X_+,k)) \cong H^*(X;k)$, we see that the $E_2$-pages of these spectral sequences agree. The filtrations are identified as well, as both can be viewed as being derived from the cobar construction. Therefore, if both spectral sequences converge strongly, then we obtain an equivalence between their targets, which is what we desired. Luckily, the K{\"u}nneth spectral sequence is always strongly convergent, and a theorem of Shipley \cite{shipleysseq} guarantees convergence of the Eilenberg--Moore spectral sequence if the hypotheses of \cref{rognesgalois} are satisfied.


We apply \cref{rognesgalois} to the fibre sequence $G / H \to BH \to BG$, where $G$ is a cyclic $p$-group or a generalised quaternion group, and $H$ is its elementary abelian $p$-subgroup. Notice that $BG$ is path-connected, and $G$, as a $p$-group, acts nilpotently on $H_*(G/H;k)$. We deduce that we have a $G/H$-Galois extension $F(BG_+,k) \to F(BH_+,k)$. Now since $G$ acts trivially on~$k$, these function spectra are naturally identified with homotopy fixed points, so we have proved the following results:

\begin{thm}
	\label{thm:khCp}
	The natural ring map $k^{hC_{p^n}} \to k^{hC_p}$ is a $C_{p^{n-1}}$-Galois extension of $\mathbb{E}_\infty$-rings.
\end{thm}

\begin{thm}
\label{thm:khq}
	The natural ring map $k^{hQ_{2^n}} \to k^{hC_2}$ is a $Q_{2^n}/C_2$-Galois extension of $\mathbb{E}_\infty$-rings.
\end{thm}

We proceed to use this to show that we have Galois extensions $k^{tG} \to k^{tH}$ for all aforementioned $G$ and $H$. First, observe that given a Galois extension $R \to S$ and a map of ring spectra $R \to Q$, we can take the pushout along these maps to form the base-change $Q \to S \otimes_R Q$.  The following result of Rognes provides conditions for this map to be a Galois extension.
    
\begin{lem}[{\cite[Section 7.1]{rognes}}]
    \label{thm:basechange}
    $G$-Galois extensions are stable under dualizable base-change. Faithful $G$-Galois extensions are stable under arbitrary base-change.  
\end{lem}

From material that we discuss in the appendix, it turns out that in the cases we consider, one can identify the Tate constructions $k^{tG} \to k^{tH}$ as a (dualizable) base-change of $k^{hG} \to k^{hH}$. Indeed, work of Greenlees (\cref{thm:greenleestate}) allows us to view $k^{tG}$ as a localization of $k^{hG}$ away from the augmentation ideal $I$. This construction depends only on the radical of the ideal $I$. Now, since the groups $C_p$, $C_{p^n}$, and $Q_{2^n}$ are Cohen--Macualay, $\pi_{-*}(k^{hG})$ is free over a polynomial subalgebra $A \cong k[x]$.  The radical of the ideal $(x)$ is the same as the radical of the augmentation ideal $I$, and so we obtain the following pushout diagrams of ring spectra.
\begin{center}
    \begin{tikzcd}[column sep = huge, row sep = huge]
k^{hC_{p^n}} \arrow[d] \arrow[r]  \arrow[dr, phantom, "\ulcorner", very near end] & k^{hC_p} \arrow[d] \\
k^{tC_{p^n}} \arrow[r]           & k^{tC_p}          
\end{tikzcd}
\qquad\text{and} \qquad
\begin{tikzcd}[column sep = huge, row sep = huge]
k^{hQ_{2^n}} \arrow[d] \arrow[r] \arrow[dr, phantom, "\ulcorner", very near end] & k^{hC_2} \arrow[d] \\
k^{tQ_{2^n}} \arrow[r] & k^{tC_2}          
\end{tikzcd}
\end{center}
In both cases the Tate constructions $k^{tG}$ and $k^{tH}$ are both formed by the same finite localizations (away from the ideal $(x)$).  Since $k^{tG}$ can be identified as a finite localization of $k^{hG}$, it is therefore dualizable.  The discussion above therefore proves the following theorems:

\begin{thm}
	\label{thm:ktCp}
	The natural ring map $k^{tC_{p^n}} \to k^{tC_p}$ is a $C_{p^{n-1}}$-Galois extension of $\mathbb{E}_\infty$-rings.
\end{thm}

\begin{thm}
\label{thm:ktq}
	The natural ring map $k^{tQ_{2^n}} \to k^{tC_2}$ is a $Q_{2^n}/C_2$-Galois extension of $\mathbb{E}_\infty$-rings.
\end{thm}

\begin{rmk}
This approach holds more generally than for the groups we consider in this paper.  For example, this argument in fact shows that if $A$ is an abelian $p$-group with maximal elementary subgroup $(C_p)^n$, then the natural ring map $k^{tA} \to k^{t(C_p)^n}$ is a $A/(C_p)^n$-Galois extension of $\mathbb{E}_\infty$-rings.

The key idea is to show that the image under the restriction map of the radical of the augmentation ideal $I$ in $k^{hG}$ is the radical of the augmentation ideal $J$ in $k^{hH}$.  

This is essentially the idea of Quillen's $\mathcal{F}$-isomorphism theorem \cite[Theorem 7.1]{Quillen71}, which discusses how for a certain family of elementary abelian subgroups $\mathcal{E}_p(G)$, the restriction map $$H^*(G;k) \to \prod_{E \in \mathcal{E}_p(G)} H^*(E;k)$$ is an isomorphism up to radicals.  Determining to which extent (e.g. the precise groups for which this method holds) is current work in progress.
\end{rmk}

\begin{rmk}
    It is interesting to compare these results to Mathew's \cite[Thm.~9.17]{mathew}, where it's proved that the {\'e}tale fundamental group of $\StMod(kG)$ identifies with the profinite completion of $\ho_1 |\mc{O}_{\mc{A}}(G)|$, which of course simplifies to $G/H$ if $\mc{A} = \langle H\rangle$. In particular, $k^{tG} \to k^{tH}$ can be viewed as the universal cover of~$k^{tG}$.
\end{rmk}

\subsection{Faithful Galois extensions}
\label{faithfulgaloisextensions}

Our next goal is to prove that the Galois extensions $k^{tC_{p^n}} \to k^{tC_p}$ and $k^{tQ_{2^n}}\to k^{tC_n}$ are faithful.  We will repeatedly invoke the following criterion of Rognes to show that our Galois extensions are faithful.

\begin{thm}[{\cite[Prop. 6.3.3]{rognes}}]
\label{thm:rognesfaithful}
	A $G$-Galois extension $f \colon R \to S$ is faithful if and only if the Tate construction $S^{tG}$ is contractible.
\end{thm}

This is especially useful because of the existence of the multiplicative \textbf{Tate spectral sequence}: if~$X$ is a spectrum with a $G$-action for some group $G$, then there is a spectral sequence
\[E_2^{st} = \widehat{H}^s\big(G; \pi_t(X)\big) \Rightarrow \pi_{t-s}(X^{tG})\,\text{,}\]
with differentials $d_r\colon E_r^{s,t} \to E_r^{s+r,t+r-1}$, which lets us compute the homotopy groups of $X^{tG}$ in terms of more readily accessible Tate cohomology groups. Moreover, we can leverage naturality and the cofibre sequence $X_{hG} \to X^{hG} \to X^{tG}$ to compare and import many differentials between the homotopy orbit spectral sequence, the homotopy fixed point spectral sequence, and the Tate spectral sequence.

\begin{rmk}
We remark that it suffices to show that $k^{hG} \to k^{hH}$ is a faithful $G/H$-Galois extension.  By \cref{thm:basechange}, this implies that Galois descent holds for $k^{tC_{p^n}} \to k^{tC_p}$ and $k^{tQ_{2^n}} \to k^{tC_2}$. However, the proof would proceed in exactly the same way (i.e. computing the Tate spectral sequence).  Furthermore, we require the HFPSS calculations involving  $k^{tC_{p^n}} \to k^{tC_p}$ and $k^{tQ_{2^n}} \to k^{tC_2}$ in our Picard spectral sequence calculations.  
\end{rmk}


\subsection{The case of cyclic $p$-groups}
\label{section:cyclicdescent}
\label{ss:Cpfaithful}

\begin{thm}
	\label{thm:ktCpfaithful}
	The $C_{p^{n-1}}$-Galois extension $k^{tC_{p^n}} \to k^{tC_p}$ of $\mathbb{E}_\infty$-rings is faithful.
\end{thm}

\begin{proof}
    Our goal is to compute the Tate spectral sequence 
	\[ E_2^{st} \cong \wh{H}^{s}\big(C_{p^{n-1}};\pi_t k^{tC_p}\big) \Ra \pi_{t - s} (k^{tC_p})^{tC_{p^{n-1}}}\]
	and show that the Tate spectrum $(k^{tC_p})^{tC_{p^{n-1}}}$ is contractible. To do so, we recall that the natural map $k^{hG} \to k^{tG}$ from homotopy fixed points to Tate fixed points allows us to import differentials from the HFPSS
	\[E_2^{st} = H^s\big(C_{p^{n-1}};\pi_t k^{hC_p}\big) \Ra \pi_{t - s}(k^{hC_{p}})^{hC_{p^{n-1}}}\,\text{.}\]
	Note that $(k^{hC_{p}})^{hC_{p^{n-1}}} \simeq k^{hC_{p^n}}$. In fact, this spectral sequence can be identified with the Lyndon--Hochschild--Serre spectral sequence associated to the (central) extension $C_p \to C_{p^{n}} \to C_{p^{n-1}}$, which is well-understood. We review this spectral sequence, distinguishing between the cases where $p = 2$ and $p$ is odd.
	
	If $p = 2$, then the $E_2$-page of the Hochschild--Serre spectral sequence, depicted in \cref{E2lhsC2}, is of the form
		\[E_2^{st} \cong H^s(C_{2^{n-1}};k) \otimes H^{-t}(C_2;k) \cong \left\{
		\begin{array}{ll}
		k[x_1] \otimes k[t_1] & n = 2 \\
		k[x_2] \otimes \Lambda(x_1) \otimes k[t_1] & n \geq 3
		\end{array}
		\right.\]
	Here, $x_1$ is of Adams degree $(-1,1)$, $x_2$ is of degree $(-2,2)$, and $t_1$ is of Adams degree $(-1,0)$. A standard argument shows that $d_2(t_1)$ is nontrivial, whereas $d_2$ vanishes on the remaining generators for degree reasons. By multiplicativity, this determines the remaining differentials. The $E_3$-page has been illustrated in \cref{E3lhsC2}, and the spectral sequence collapses.
	
	\begin{figure}
        \centering
        \includegraphics{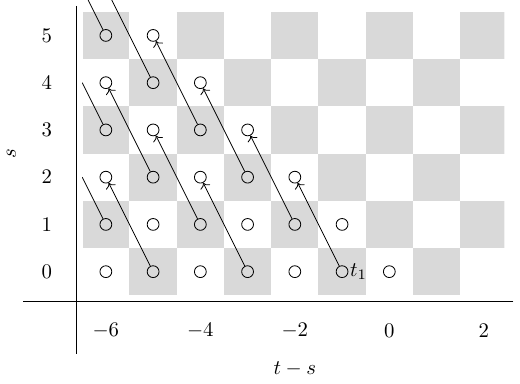}
            \caption{The Adams-graded $E_2$-page of the Hochschild--Serre spectral sequence associated to the extension $C_2 \to C_{2^n} \to C_{2^{n-1}}$.  The circles denote a $k$-summand, and the nonzero differentials have been drawn.}
        \label{E2lhsC2}
    \end{figure}
    \begin{figure}
        \centering
        \includegraphics{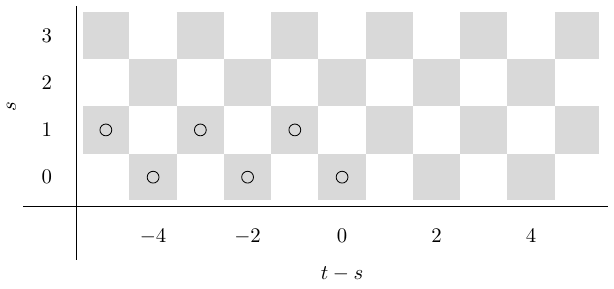}
            \caption{The $E_3$-page of the Hochschild--Serre spectral sequence associated to the extension $C_2 \to C_{2^n} \to C_{2^{n-1}}$. There are no remaining differentials, and the spectral sequence collapses.}
        \label{E3lhsC2}
        \end{figure}
    
    We can now leverage this information to the HFPSS computing $\pi_{*} (k^{tC_{2^n}})$. Recall from \cref{tatereference}
    that the ring $\pi_*(k^{tC_2})$ is isomorphic to $k[t_1^{\pm1}]$, so the $E_2$-page, illustrated in \cref{E2ktCpn}, is now given by
    \[E_2^{s,t} \cong H^s\big(C_{2^{n-1}}; \pi_t(k^{tC_2})\big) \cong \left\{
		\begin{array}{ll}
		k[x_1] \otimes k[t_1^{\pm1}] & n = 2 \\
		k[x_2] \otimes \Lambda(x_1) \otimes k[t_1^{\pm1}] & n \geq 3
		\end{array}
		\right.\]
	By multiplicativity, we can simply extend the differentials of \cref{E2lhsC2} to negative powers of $t_1$ using the Leibniz rule. The $E_3$-page has been drawn out in \cref{E3tateC2}, where it collapses.
    
    \begin{figure}
    \centering
    \includegraphics{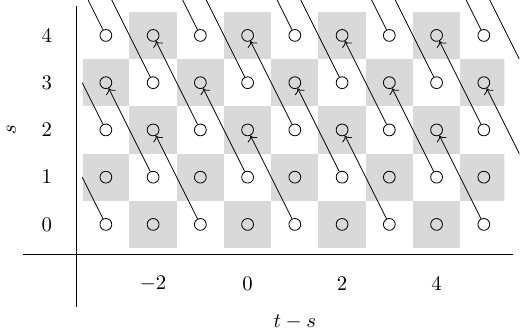}
    \caption{The Adams-graded $E_2$-page of the HFPSS computing the homotopy groups $\pi_*\big(k^{tC_{2^n}}\big)$ for $n \geq 2$. It is effectively just \cref{E2lhsC2} extended to another quadrant.}
    \label{E2ktCpn}
\end{figure}

\begin{figure}
    \centering
    \includegraphics{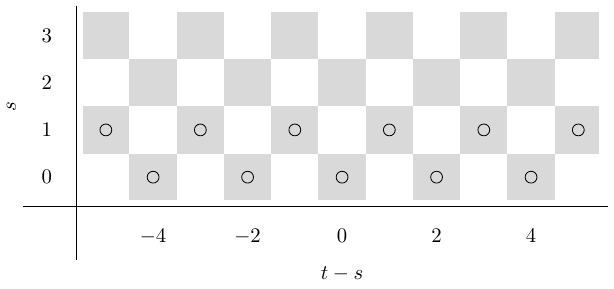}
    \caption{The $E_3$-page of the HFPSS computing $\pi_*(k^{tC_{2^n}})$. There are no remaining differentials, and the spectral sequence collapses.}
    \label{E3tateC2}
\end{figure}
    
    We now use this information to compute the Tate spectral sequence. From \cref{tatereference} we see that passing to Tate cohomology again amounts to inverting the relevant generators on cohomology (namely, $x_1$), and so we simply take the differentials of the HPFSS, and extend to negative $s$-degree by multiplicativity. The $E_2$-page has been drawn in \cref{tcp-E2tatess}. We see that every summand is now killed by a nontrivial differential. The $E_3$-page is therefore empty, and the Tate construction is contractible. By \cref{thm:rognesfaithful}, we have therefore shown that $k^{tC_{2^n}} \to k^{tC_2}$ is a faithful Galois extension.
    
    \begin{figure}
    \centering
    \includegraphics{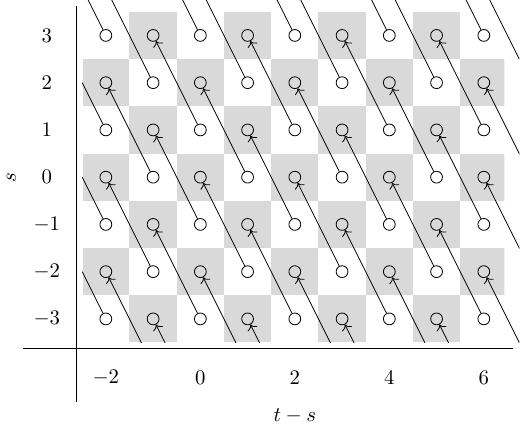}
    \caption{\label{tcp-E2tatess}The Adams-graded $E_2$-page of the Tate spectral sequence computing the homotopy groups $\pi_*\big((k^{tC_p})^{tC_{p^{n-1}}}\big)$ for $n \geq 2$.}
    \end{figure}
    
    If $p$ is odd, the proof technique is the same, though the multiplicative structure of the Hochschild--Serre spectral sequence changes. One now has
    \[E_2^{st} \cong H^s(C_{p^{n-1}};k) \otimes H^{-t}(C_p;k) \cong k[x_2] \otimes \Lambda(x_1) \otimes k[t_2] \otimes \Lambda(t_1)\,\text{,}\]
    with nontrivial differential $d_2(t_1) = x_2$. The $E_2$-page looks identical to \cref{E2lhsC2}, and the $E_3$-page to \cref{E3lhsC2}.
    
    As before, by multiplicativity we can extend this to positive $(t-s)$-degree into the HFPSS for Tate spectra.  Similarly, we can then further extend to negative $s$-degree into the Tate spectral sequence. The $E_2$-page of the latter looks identical to \cref{tcp-E2tatess}, and we conclude that the Tate construction is again contractible. Therefore, $k^{tC_{p^n}} \to k^{tC_p}$ is a faithful Galois extension.
\end{proof}    

\subsection{The case of the quaternion group}
\label{section:galoisquaternion}

\begin{thm}
\label{thm:ktq8faithful}
	The natural ring map $k^{tQ_8} \to k^{tC_2}$ is a \textit{faithful} $(C_2)^2$-Galois extension of $\mathbb{E}_\infty$-rings.
\end{thm}
The reason we treat the $Q_8$ case separately from the generalized quaternion case $Q_{2^n}$ is because the group cohomology and Tate cohomology rings differ between these two cases, as do the resulting differentials.

\begin{proof}
    Our method is the same as in the cyclic $p$-group case: we first study the Hochschild--Serre spectral sequence associated to the extension $C_2 \to Q_{8} \to (C_2)^2$, which can be identified with the HFPSS computing $\pi_*(k^{hQ_8})$.  We then leverage multiplicativity twice to compute the Tate spectral sequence
	\[ E_2^{st} \cong \wh{H}^{s}\big((C_2)^2;\pi_t k^{tC_2}\big) \Ra \pi_{t - s} (k^{tC_2})^{t(C_2)^2}\,\text{.}\]

    \begin{figure}
    \centering
    \includegraphics{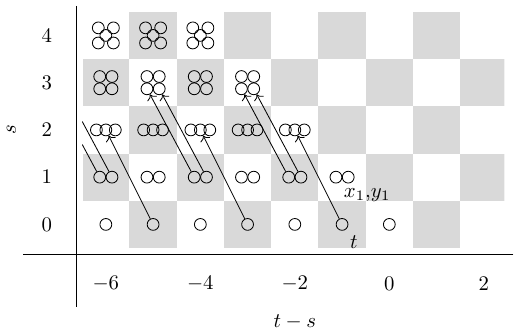}
        \caption{$E_2$-page of the Hochschild--Serre spectral sequence, or equivalently, the $(C_2)^2$-HFPSS, associated to the extension $C_2 \to Q_8 \to (C_2)^2$. To prevent cluttering, we have only illustrated the nonzero differentials for small~$s$. Each circle represents a $k$-summand.}
        \label{hq-E2hfpss}
    \end{figure}
    
    The Hochschild--Serre spectral sequence, regraded to match with the grading conventions of the HFPSS, has $E_2$-page of the form
	\[E_2^{s,t} \cong H^{s}\big((C_2)^2; k\big) \otimes \pi_t(k^{hC_2})\cong k[x_1, y_1] \otimes k[t_1]\,\text{.}\]
	where $x_1$ and $y_1$ are in Adams degree $(-1,1)$ and $t_1$ is in Adams degree $(-1,0)$. To understand the differentials, one can restrict to appropriate subgroups of $Q_8$, which yield natural maps of extensions.  For example, one has \vspace{0mm}
	\[ \begin{tikzcd}[column sep = huge, row sep = huge]
		1 \ar[r] & C_2 \ar[r] \ar[d] & C_4 \ar[r] \ar[d] & C_2 \ar[d] \ar[r] & 1 \\
		1 \ar[r] & C_2 \ar[r] & Q_8 \ar[r] & C_2 \times C_2 \ar[r] & 1 \end{tikzcd}\]
    These extentions induce comparison maps of Hochschild--Serre spectral sequences for $Q_8$ and for $C_4$. The spectral sequence for $C_4$ had been outlined in the previous section, and we infer that $d_2(t_1) = x_1^2 + x_1y_1 + y_1^2$. The $E_2$-page has been drawn out in \cref{hq-E2hfpss}. By e.g. Kudo transgression one then finds that $d_3(t_1^2) = \name{Sq}^1\big(d_2(t_1)\big) = x_1^2 y_1 + x_1 y_1^2$. The $E_3$-page and $E_4$-page have been drawn in \cref{E3lhsQ8} and \cref{hq-E4hfpss}, respectively. Observe that the spectral sequence collapses on the $E_4$-page.
    
\begin{figure}
\centering
    \includegraphics{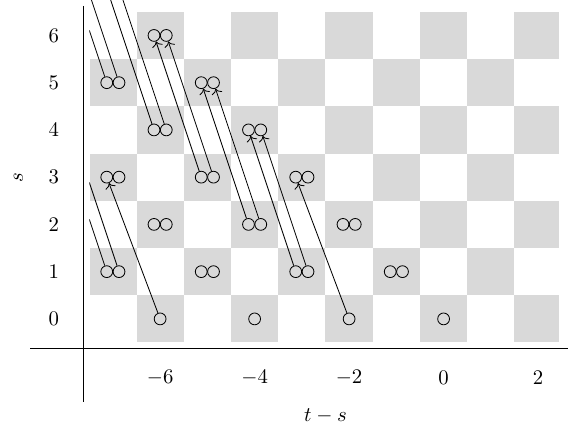}
    \caption{$E_3$-page of the Hochschild--Serre spectral sequence for $C_2 \to Q_8 \to (C_2)^2$.}
\label{E3lhsQ8}
\end{figure}

\begin{figure}
    \centering
    \includegraphics{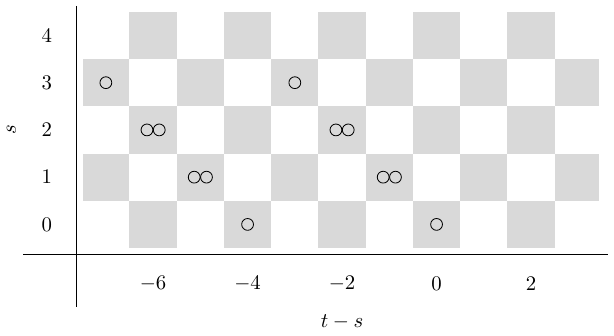}
    \caption{\label{hq-E4hfpss} $E_4$-page of the Hochschild--Serre spectral sequence for $C_2 \to Q_8 \to (C_2)^2$. There are no remaining differentials, and the spectral sequence collapses.}
\end{figure}
	
We can now again compute $\pi_{*} (k^{tQ_8})$ via the HFPSS. Since the ring structure of $\pi_*(k^{tC_2})$ is given simply by $k[t_1^{\pm 1}]$, we may again import all differentials from the Hochschild--Serre spectral sequence and extend using multiplicativity. The $E_2$-page is illustrated in \cref{tq-E2hfpss}. It develops in the expected way: the $E_3$-page is outlined in \cref{E3tateQ8}, and the $E_4$-page in \cref{tq-E4hfpss}.

\begin{figure}
\centering
\includegraphics{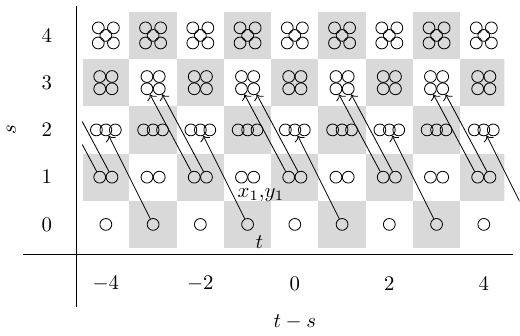}
    \caption{$E_2$-page of the $(C_2)^2$-HFPSS computing $\pi_*(k^{tQ_8})$. It is obtained from \cref{hq-E2hfpss} by inverting $t$.}
    \label{tq-E2hfpss}
\end{figure}

\begin{figure}
\centering
\includegraphics{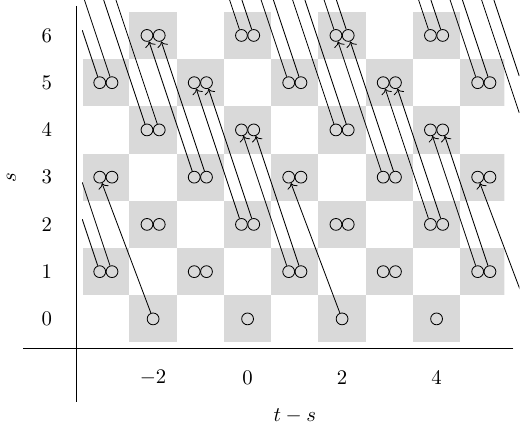}
    \caption{$E_3$-page of the $(C_2)^2$-HFPSS computing $\pi_*(k^{tQ_8})$.}
\label{E3tateQ8}
\end{figure}

\begin{figure}
\centering
\includegraphics{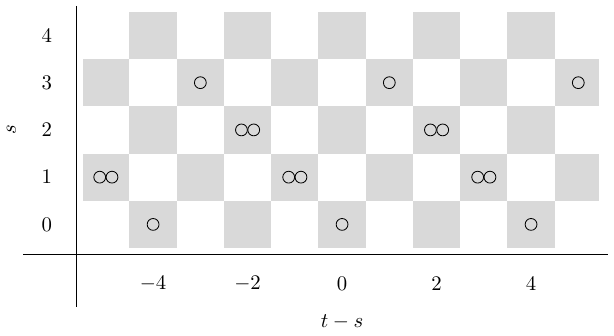}
    \caption{\label{tq-E4hfpss}
    $E_4$-page of the $(C_2)^2$-HFPSS computing $\pi_*(k^{tQ_8})$. Compare with \cref{hq-E4hfpss}.}
\end{figure}

We further extend to negative $s$-degree so as to obtain the Tate spectral sequence
\[ E_2^{st} \cong \wh{H}^{-s}\big((C_2)^2,\pi_{t-s}(k^{tC_2})\big) \Ra \pi_{t-s}\big((k^{tC_2})^{t(C_2)^2}\big)\,\text{.}\]
Here, some care must be taken in extending to the Tate spectral sequence, as the Tate cohomology ring of $(C_2)^2$ isn't just given by a naive Laurent polynomial ring. As computed in \cref{appendixtate}, the multiplicative structure in \textit{nonnegative} degree is identified with that of the cohomology ring. But in negative degrees, we have the following. There is a distinguished element $\alpha$ in $\wh{H}^{-1}\big((C_2)^2;k\big)$, and the cup product yields a perfect pairing $\wh{H}^{r}\big((C_2)^2;k\big) \otimes \wh{H}^{-r-1}\big((C_2)^2;k\big) \to \wh{H}^{-1}\big((C_2)^2;k) \cong \langle \alpha\rangle$. The remaining cup products, in particular all products of negative-degree elements, are zero. In view of the perfect pairing, we denote the negative-degree classes by $\alpha x_1^{-a} y_1^{-b}$, though it is not a cup product of $\alpha$ by some element $x_1^{-a} y_1^{-b}$. 

It is thanks to the pairing that we can extend the differentials to negative $s$-degree. For instance, we have
\[\begin{split}
    d_2(\alpha x_1^{-a} y_1^{-b} \otimes t_1^{-1}) &= d_2(t_1^{-1}) \cdot \alpha x_1^{-a} y_1^{-b} \\
    &= \alpha x_1^{2-a} x_2^{-b} + \alpha x_1^{1-a} x_2^{1-b} + \alpha x_1^{-a} x_2^{2-b}
\end{split} \]
We've drawn the $E_2$-page on \cref{tq-E2tatess}. The $E_3$-page and $E_4$-page of the Tate spectral sequence have been drawn in \cref{tq-E3tatess} and \cref{tq-E4tatess}.

\begin{figure}
\centering
\includegraphics{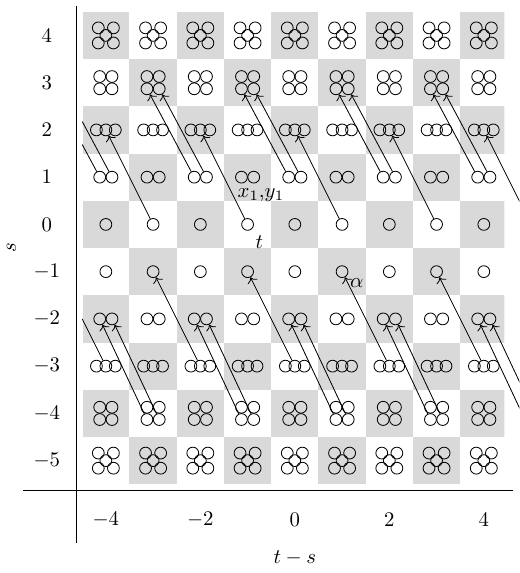}
\caption{$E_2$-page of the Tate spectral sequence computing the homotopy groups $\pi_*\big((k^{tC_2})^{t(C_2)^2}\big)$. To prevent cluttering, only the nonzero differentials for small $s$ are drawn.}
\label{tq-E2tatess}
\end{figure}

\begin{figure}
\centering
\includegraphics{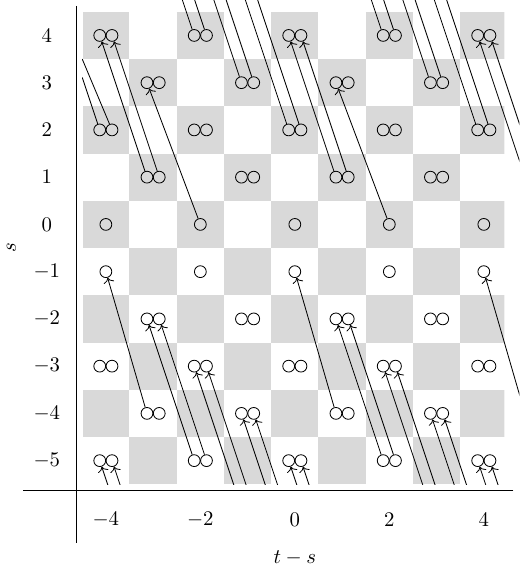}
\caption{$E_3$-page of the Tate spectral sequence computing the homotopy groups $\pi_*\big((k^{tC_2})^{t(C_2)^2}\big)$. All nonzero differentials have been illustrated.}
\label{tq-E3tatess}
\end{figure}

\begin{figure}
\centering
\includegraphics{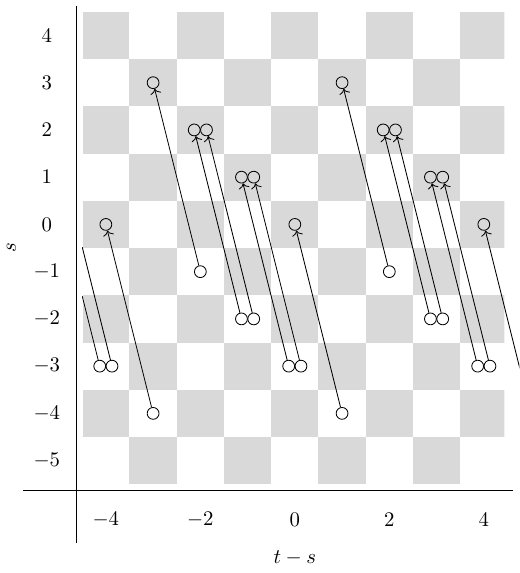}
\caption{$E_4$-page of the Tate spectral sequence computing the homotopy groups $\pi_*\big((k^{tC_2})^{t(C_2)^2}\big)$. These nontrivial differentials do not come from the HFPSS.}
\label{tq-E4tatess}
\end{figure}

In the HFPSS, the spectral sequence collapses at the $E_4$-page for degree reasons, but in the Tate spectral sequence, there's room for nontrivial $d_4$-differentials. We claim that these differentials are indeed nontrivial.  We begin with the following square of cofiber sequences.
\[
\begin{tikzcd}[column sep = large, row sep = large]
(k^{tC_2})_{h(C_2)^2} \arrow[r]           & (k^{tC_2})^{h(C_2)^2} \arrow[r]           & (k^{tC_2})^{t(C_2)^2}           \\
(k^{hC_2})_{h(C_2)^2} \arrow[r] \arrow[u] & (k^{hC_2})^{h(C_2)^2} \arrow[r] \arrow[u] & (k^{hC_2})^{t(C_2)^2} \arrow[u] \\
(k_{hC_2})_{h(C_2)^2} \arrow[r] \arrow[u] & (k_{hC_2})^{h(C_2)^2} \arrow[r] \arrow[u] & (k_{hC_2})^{t(C_2)^2} \arrow[u]
\end{tikzcd}
\]
We can identify the middle term as $k^{hQ_8}$ and the bottom left term as $k_{hQ_8}$. Moreover, thanks to \cref{thm:ktq} we can identify the top middle term with $k^{tQ_8}$. This simplifies the diagram to
\[
\begin{tikzcd}[column sep = large, row sep = large]
(k^{tC_2})_{h(C_2)^2} \arrow[r]           & k^{tQ_8} \arrow[r]           & (k^{tC_2})^{t(C_2)^2}           \\
(k^{hC_2})_{h(C_2)^2} \arrow[r] \arrow[u] & k^{hQ_8} \arrow[r] \arrow[u] & (k^{hC_2})^{t(C_2)^2} \arrow[u] \\
k_{hQ_8} \arrow[r] \arrow[u]     & k_{hQ_8} \arrow[r] \arrow[u] & 0 \arrow[u]
\end{tikzcd}
\]
This forces the map $(k^{hC_2})^{t(C_2)^2} \to (k^{tC_2})^{t(C_2)^2}$ to be an isomorphism. Now, to both we may functorially asssociate a Tate spectral sequence. The $E_4$-page of $(k^{tC_2})^{t(C_2)^2}$ has been illustrated in \cref{tq-E4tatess}, and that of $(k^{hC_2})^{t(C_2)^2}$ is the same but truncated so as to live in $t$-degree $\leq 0$. The comparison map of spectral sequences is the obvious one. This comparison forces the $d_4$-differentials in \cref{tq-E4tatess} for $t - s > 1$ to be nontrivial, and by multiplicativity, this nontriviality propagates to negative $(t-s)$-degree.

Therefore, the $E_\infty$-page empty and the Tate construction $(k^{tC_2})^{t(C_2)^2}$ is contractible. Therefore, $k^{tQ_8} \to k^{tC_2}$ is a faithful Galois extension.
\end{proof}

\begin{rmk}
This argument also implies that $k^{hQ_8} \to k^{hC_2}$ is a faithful Galois extension, and by the commutative square above, we also have equivalences $(k^{hC_2})_{h(C_2)^2} \simeq (k^{hC_2})^{h(C_2)^2}$ and $(k^{tC_2})_{h(C_2)^2} \simeq (k^{tC_2})^{h(C_2)^2}$.
\end{rmk}

\subsection{The case of generalized quaternion groups}
\label{generalizedquaternionfaithful}

\begin{thm}
\label{thm:ktq2nfaithful}
	The natural ring map $k^{tQ_{2^n}} \to k^{tC_2}$ is a \textit{faithful} $Q_{2^n}/C_2$-Galois extension of $\mathbb{E}_\infty$-rings.
\end{thm}

\begin{proof}
    Our method is the same as in the previous cases. In fact, the associated spectral sequence diagrams look exactly the same as in the $Q_8$ case; the only difference is that the multiplicative structure changes.
    
    We first study the Hochschild--Serre spectral sequence associated to the extension $C_2 \to Q_{2^n} \to D_{2^{n-1}}$, and then we extend this spectral sequence to produce the four-quadrant Tate spectral sequence
	\[ E_2^{st} \cong \wh{H}^{s}\big(D_{2^{n-1}};\pi_t k^{tC_2}\big) \Ra \pi_{t - s} (k^{tC_2})^{tD_{2^{n-1}}}\,\text{.}\]
	For all $n \geq 4$, the cohomology ring $H^*(D_{2^{n-1}};k)$ is given by $k[x_1,y_1,z_2] / (x_1y_1)$, where $|x_1| = |y_1| = 1$ and $|z_2| = 2$. 
	Moreover, $\name{Sq}^1(z_2) = (x_1+y_1)z_2$.  It is convenient to set $u_1 = x_1 + y_1$ and write the cohomology ring as $k[x_1,u_1,z_2] / (u_1 x_1 + x_1^2)$. 
	
	Since $C_2$ is central in $Q_{2^n}$, the $E_2$-page of the Hochschild--Serre spectral sequence has the form
	\[ E_2^{st} \cong H^s(D_{2^{n-1}};k) \otimes \pi_t(k^{hC_2}) \cong k[x_1,u_1,z_2] / (u_1x_1 + x_1^2) \otimes k[t_1]\,\text{.}\]
	We have a nontrivial $d_2$-differential $d_2(t_1) = u_1^2 + z_2$, as can be computed by restricting to appropriate subgroups of $Q_{2^n}$, and by Kudo transgression, we have $d_3(t_1^2) = u_1 z_2$. These again spawn all the other differentials via the Leibniz rule. Although the multiplicative generators are different, the $E_2$-, $E_3$-, and $E_4$-page look exactly the same as those for $Q_8$ --- cf. \cref{hq-E2hfpss}, \cref{E3lhsQ8}, and \cref{hq-E4hfpss}. 

    We extend the spectral sequence using multiplicativity to the HFPSS computing $\pi_*(k^{tQ_{2^n}})$. Again, since $\pi_*(k^{tC_2})$ is simply $k[t_1^{\pm 1}]$, we can extend without much issue. The pages are again identical to $Q_8$, and are illustrated in \cref{tq-E2hfpss}, \cref{E3tateQ8}, and \cref{tq-E4hfpss}. We then further extend to the Tate spectral sequence. As in the $Q_8$ case, some care must be taken when extending, because the multiplicative structure of $\wh{H}^{*}(D_{2^{n-1}};k)$ is nontrivial. As shown in \cref{tatereference}, the Tate cohomology ring is the usual cohomology ring in positive degrees, and there's again a perfect pairing onto $\wh{H}^{-1}(D_{2^{n-1}};k) \cong \langle \alpha\rangle$, and we use the perfect pairing to extend the differentials to negative $s$-degree. The $E_2$- and $E_3$-page look the same as in \cref{tq-E2tatess} and \cref{tq-E3tatess}.
    
    For the same reason as in $Q_8$, there is room for nontrivial differentials on the $E_4$-page of the Tate spectral sequence. The proof that they are indeed nontrivial is exactly the same: one has the square of cofibre sequences
    \[
    \begin{tikzcd}[column sep = large, row sep = large]
    (k^{tC_2})_{hD_{2^{n-1}}} \arrow[r]           & k^{tQ_{2^n}} \arrow[r]           & (k^{tC_2})^{tD_{2^{n-1}}}           \\
    (k^{hC_2})_{hD_{2^{n-1}}} \arrow[r] \arrow[u] & k^{hQ_{2^n}} \arrow[r] \arrow[u] & (k^{hC_2})^{tD_{2^{n-1}}} \arrow[u] \\
    k_{hQ_{2^n}} \arrow[r] \arrow[u]     & k_{hQ_{2^n}} \arrow[r] \arrow[u] & 0 \arrow[u]           
    \end{tikzcd}
    \]
    which implies that the map $(k^{hC_2})^{tD_{2^{n-1}}} \to (k^{tC_2})^{tD_{2^{n-1}}}$ is an equivalence. This forces the nontriviality of some $d_4$-differentials, and the nontriviality of all other differentials then follows by multiplicativity. Thus the Tate construction is again contractible. Therefore, $k^{tQ_{2^n}} \to k^{tC_2}$ is a faithful Galois extension.
\end{proof}

\begin{rmk}
\label{rmk:abstractdescent}
Though we have provided explicit proofs that these Galois extensions are faithful, this is abstractly true through descent theory, which we now sketch.  We are grateful to Akhil Mathew for pointing out this argument.

Recall that we say that a commutative algebra object $A \in \mathcal{C}$ admits descent if the thick $\otimes$-ideal generated by $A$ is all of $\mathcal{C}$.  The following proposition follows from a standard thick tensor ideal argument.

\begin{prop}{\cite[Prop. 3.19]{mathew}}]
If $A \in \mathsf{CAlg}(\mathcal{C})$ admits descent, then $A$ is faithful.
\end{prop}

Let $\mathcal{E}_p$ be the family of elementary abelian $p$-subgroups of $G$.  For each subgroup $H \in \mathcal{E}_p$, we can form the commutative algebra object $A_H := \prod_{G/H}k \in \mathsf{CAlg}(\mathsf{Fun}(BG, \Mod(k)))$.

\begin{prop}[{\cite[Prop. 9.13]{mathew}}]
The commutative algebra object $$A := \prod_{H \in \mathcal{E}_p}\big(\prod_{G/H}k\big) \in \mathsf{CAlg}(\mathsf{Fun}(BG, \Mod(k)))$$
admits descent.
\end{prop}

Recall that we have a localization functor $\mathsf{Fun}(BG, \Mod(k)) \to \mathsf{StMod}(kG)$.  We denote the image of $A_H$ and $A$ under this functor by $\mathcal{A}_H$ and $\mathcal{A}$ respectively.  It follows that $\mathcal{A}$ is descendable in $\mathsf{StMod}(kG)$. In fact, by work of Balmer \cite{balmer} (cf. \cite{mathew} 9.12), we have an equivalence  $\Mod_{\mathsf{StMod}(kG)}(\mathcal{A}_H) \simeq \mathsf{StMod}(kH)$.  Moreover, we can identity the adjunction $\mathsf{StMod}(kG) \rightleftarrows \Mod_{\mathsf{StMod}(kG)}(\mathcal{A}_H)$ with the restriction-coinduction adjunction $\mathsf{StMod}(kG) \rightleftarrows \mathsf{StMod}(kH)$.

For the groups we consider (cyclic $p$-groups and generalized quaternion groups), there is a single elementary abelian $p$-subgroup.  Thus $\mathcal{A} = \mathcal{A}_H$, and the cobar construction for $\mathcal{A}$ exhibits the equivalence $\mathsf{StMod}(kG) \simeq  \mathsf{StMod}(kH)^{hG/H}$. In particular, this implies that the morphism $k^{tG} \to k^{tH}$ admits descent, and hence $k^{tH}$ is faithful.
\end{rmk}

\section{Computation of endotrivial modules}
\label{sectionendotrivial}

In this chapter, we will evaluate the limit spectral sequence to compute the group of endotrivial modules for the cyclic $p$-groups and generalized quaternion groups. In these cases, we saw that the limit decomposition can be re-interpreted as an instance of Galois descent. 

Accordingly, the limit spectral sequence for $\Omega$ is a familiar object. Indeed, by \cref{omegaistate}, $\Omega \StMod(kG)$ is simply $k^{tG}$, and the limit spectral sequence is simply an extension of the Hochschild--Serre spectral sequence to two quadrants. We have already evaluated this spectral sequence in \cref{sectiongaloisdescent}, under the guise of an HFPSS computing $\pi_*(k^{tG})$.

We can thus compute the homotopy groups of the Picard spectrum via the homotopy fixed point spectral sequence (see \cref{picspectralsequence}), which takes the form 
\[E_2^{st} = H^s\big(G/H;\pi_t \mf{pic}\StMod(kH)\big) \Ra \pi_{t - s}\big(\mf{pic}\,{\StMod}(kG)\big)\,\text{.}\]

Recall that we understand the $E_2$ page of the spectral sequence well by \cref{homotopygroupsofPic}. In \cref{section:cyclicdescent}, we have already computed $\pi_n(k^{tC_p}) \cong k$ for all $p$ and $n$. It remains to compute $\pi_0(\pic(k^{tC_p})) \cong \Pic(k^{tC_p})$, but this is the content of Dade's theorem (\cref{reference}). 

\begin{prop}
The homotopy groups of $\pic(k^{tC_p})$ are given by:
$$\pi_*(\pic(k^{tC_p})) \cong 
\left\{
		\begin{array}{ll}
			\Pic(k^{tC_p}) \cong 1 & * = 0, \ p = 2 \\
			\Pic(k^{tC_p}) \cong C_2 & * = 0, \ p > 2 \\
			k^\times & * = 1\\
			 k & * \geq 2
		\end{array}
		\right.
$$    
\end{prop}

Moreover, the identification of the limit spectral sequence directly brings at our disposal the comparison tools of Mathew--Stojanoska, \cref{comparisontool} and \cref{unstabledifferential}, which allows us to import differentials from the analogous spectral sequence computing the homotopy groups of $\Omega \mc{C} \coloneqq \End(\mb{1}_{\mc{C}})$. In view of this, we will see that most of the work which remains will be to compute some unstable differentials.

\subsection{The case of cyclic $p$-groups}
\label{sectioncyclicpgroups}

In this section, our aim is to compute the Picard group of $\mi{StMod}\big(kC_{p^n}\big)$. The limit spectral sequence of \cref{picspectralsequence} reads
\[E_2^{st} = H^s(C_{p^{n-1}};\pi_t \mf{pic}\,\mi{StMod}(kC_p)\big) \Ra \pi_{t-s} \mf{pic}\,\mi{StMod}\big(kC_{p^n}\big)\,\text{.}\]
Because the groups involved are abelian, all conjugation actions are trivial, hence so is the action of $C_{p^{n-1}}$ on the $\pi_t$. The $E_2$-page has been sketched in \cref{E2limitCpn}. Let's take a look at the differentials, distinguishing between the cases $p = 2$ and odd $p$.

\begin{figure}
\centering
\includegraphics{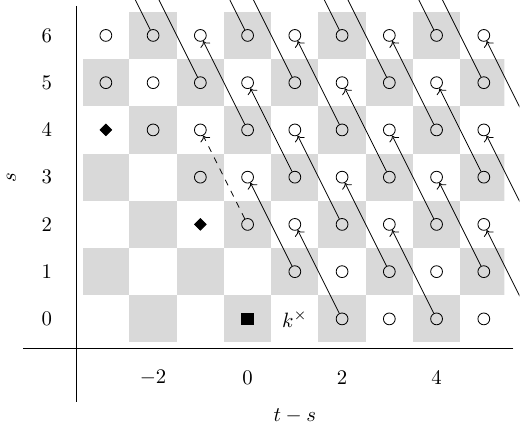}
    \caption{$E_2$-page of the limit spectral sequence for the Picard spectrum of $\mi{StMod}\big(kC_{p^n}\big)$. The circles denote a $k$-summand again. The black square is $0$ if $p = 2$ and $C_2$ if $p$ is odd. The black diamond is the group $k^\times / (k^\times)^{p^{n-1}}$. The known nonzero differentials have been illustrated. The dashed differential is of special interest, as it falls within the range of \cref{unstabledifferential}.}
\label{E2limitCpn}
\end{figure}

Let's start with the case $p = 2$. Differentials in the stable range may be compared with the differentials of the multiplicative spectral sequence
\[ E_2^{st} = H^s(C_{p^{n-1}};\pi_t \Omega \mi{StMod}(kC_{p})\big) \Ra \pi_{t-s}\Omega \mi{StMod}\big(kC_{p^n}\big)\]
using \cref{comparisontool}. But this spectral sequence we have evaluated in \cref{ss:Cpfaithful}: the $E_2$-page and $E_3$-page are sketched in \cref{E2ktCpn} and \cref{E3lhsC2}.

The only relevant differentials which remain are $d_2^{01}$ and $d_2^{22}$. The former is zero, because $k^\times$ has no $2$-torsion. (In addition, we know that the $1$-line should have a surviving $k^\times$ anyhow.) The latter may be understood via \cref{unstabledifferential}. The corresponding differential $d_2^{21}(\Omega)$ of the spectral sequence for $\Omega \mi{StMod}\big(kC_{2^n}\big)$ was the linear map $\langle t_1^{-1} x_1^2\rangle \to \langle t_1^{-2} x_1^4\rangle$ sending $t_1^{-1} x_1^2$ to $t_1^{-2} x_1^4$. Consequently, \cref{unstabledifferential} tells us $d_2^{22}$ in the limit spectral sequence for Picard spectra is the map sending a scalar $\alpha$ in $k$ to $\alpha + \alpha^2$. The kernel of this map is given by the elements $\alpha$ such that $\alpha + \alpha^2 = \alpha(\alpha + 1) = 0$, of which there's only two, namely $0$ and $1$. Therefore the kernel is $C_2$, irrespective of the underlying field $k$. 

The $E_3$-page is now summarised in \cref{E3limitC2n}. It's easily seen that, from the $0$-line upward, no nontrivial differentials can exist, and we deduce the following.

\begin{figure}
\centering
\includegraphics{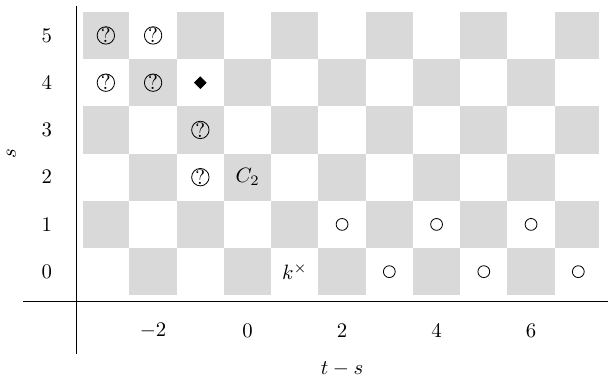}
    \caption{$E_3$-page of the limit spectral sequence for $\mi{StMod}\big(kC_{2^{n}}\big)$. Notice that the $0$-line has only one nonzero group remaining. The black diamond is the group $E_3^{43}$, which is the quotient of $k$ by the subgroup of those $c$ for which the equation $x^2 + x + c$ has a root in $k$. Classes indicated by a question mark have unknown value, as they fall outside the range where we can understand the differentials.}
\label{E3limitC2n}
\end{figure}

\begin{thm}
	\label{outcomeC2n}
	For all fields $k$ of characteristic $2$, and all $n \geq 2$, the Picard group of $\mi{StMod}\big(kC_{2^n}\big)$ is isomorphic to~$C_2$.
\end{thm}

We now turn to the case where $p$ is odd. Several minor differences arise.
\begin{itemize}[noitemsep, topsep = 0pt]
	\item The Picard group of $\mi{StMod}\big(kC_p\big)$ is $C_2$ rather than $0$ if the prime $p$ is odd.
	\item As observed in \cref{ss:Cpfaithful}, the cup product structure on $H^*(C_{p^{n-1}};k)$ is different.
	\item The squaring operation in \cref{unstabledifferential} dies in the context of odd characteristic, which alters the outcome of the Adams-graded $(0,2)$-position of the spectral sequence.
\end{itemize}
The second point causes the odd-prime analogue of the $E_2$-page of the Hochschild--Serre spectral sequence to have different multiplicative generators, but as we found in \cref{ss:Cpfaithful}, both the $E_2$-page and $E_3$-pages of the Hochschild--Serre spectral sequence look exactly the same as the $p = 2$ case --- see \cref{E2ktCpn} and \cref{E3lhsC2}.

To compute the Picard spectral sequence for $p$ odd in \cref{E2limitCpn}, we can again import differentials in the stable range.  It remains to study the unstable differentials. As before, $d_2^{01}$ is necessarily trivial. $d_2^{00}$ is trivial as well, because $k^\times / (k^\times)^{p^{n-1}}$ has no $2$-torsion, and so the $C_2$ in $E_2^{00}$ survives. The differential $d_2^{22}$ is again governed by \cref{unstabledifferential}. Since we're in odd characteristic, the squaring operation vanishes, and the differential $d_{2}^{22}$ is identified with the corresponding differential $d_{2}^{21}(\Omega)$ of \cref{E2ktCpn}, which is seen to be an isomorphism $k \to k$, and hence $E_{3}^{22}$ is $0$ rather than $C_2$. 

The $E_3$-page is summarised in \cref{E3limitCpn}. As before, there are no more nontrivial differentials that can alter the outcome, and we deduce the following result.

\begin{figure}
\centering
\includegraphics{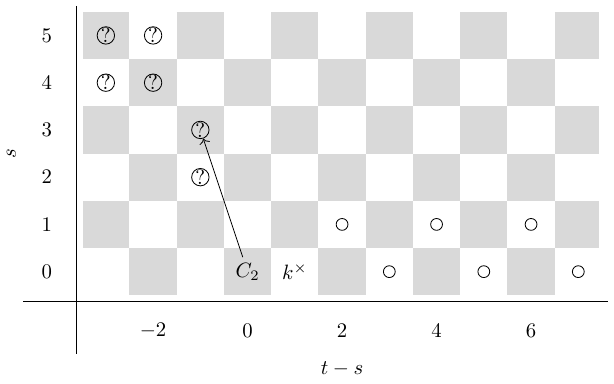}
    \caption{$E_3$-page of the limit spectral sequence for $\mi{StMod}\big(kC_{p^n}\big)$ for $p$ odd. Compare with \cref{E3limitC2n}. The only differential with possibly nontrivial domain and codomain is $d_{3}^{0,0}$, but this differential must be $0$, as $E_3^{3,2}$, arising as a subgroup of $E_2^{2,2} \cong k$, has no $2$-torsion.}
\label{E3limitCpn}
\end{figure}

\begin{thm}
	\label{outcomeCpn}
	For all fields $k$ of odd characteristic $p$ and all $n \geq 2$, the Picard group of $\mi{StMod}\big(kC_{p^n}\big)$ is isomorphic to~$C_2$.
\end{thm}

\subsection{The case of the quaternion group}
\label{sectionQ8}

We will now calculate the Picard spectral sequence 
\[ E_{2}^{st} = H^s\big((C_2)^2; \pi_t \mf{pic} \,\mi{StMod}(kC_2)\big) \Ra \pi_{t - s} \mf{pic} \,\mi{StMod}(kQ_8)\,\text{.}\]
The $E_2$-page has been illustrated in \cref{E2limitQ8}. The terms for $t \geq 2$ are cohomology groups, which we have also encountered in \cref{section:galoisquaternion}. As for $t = 1$, we notice that the crucial term $E_2^{11} = H^1\big((C_2)^2;k^\times\big)$ is zero; indeed, there are no nontrivial maps $(C_2)^2 \to k^\times$ because $k^\times$ never has any $2$-torsion.

\begin{figure}
\centering
\includegraphics{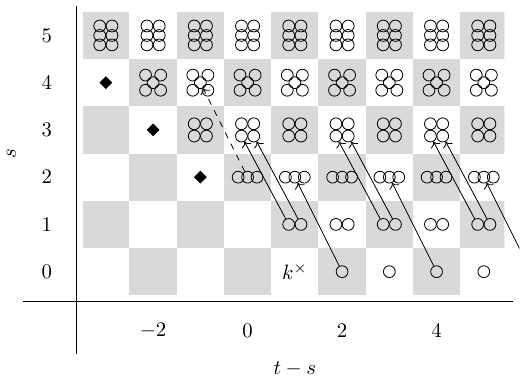}
    \caption{$E_2$-page of the limit spectral sequence for computing the Picard group of $\mi{StMod}\big(kQ_8\big)$. The nontrivial differential has been illustrated. The group indicated by the black diamond is trivial if $k$ is a perfect field. Known nonzero diagonals have been illustrated only for small $s$ to prevent cluttering; the dashed diagonal is governed by \cref{unstabledifferential}.}
\label{E2limitQ8}
\end{figure}

Using \cref{comparisontool}, the differentials in the stable range may be directly imported from the HFPSS computing $\pi_*(k^{tQ_8})$, which we have investigated in \cref{section:galoisquaternion}. The $E_2$-page, illustrated in \cref{tq-E2hfpss}, was given by
\[ E_2^{st} \cong k[x_1,y_1] \otimes k[t_1]\,\text{,}\]
with $d_2(t_1) = x_1^2 + x_1 y_1 + y_1^2$ and $d_3(t_1^2) = x_1^2 y_1 + x_1 y_1^2$.

There is an unstable differential $d_2^{22}(\pic)$, which by \cref{unstabledifferential} we may compare to $d_{2}^{21}(\Omega)$. The differential $d_2^{21}(\Omega)$ of the Hochschild--Serre spectral sequence is a $k$-linear map $k^3 \to k^5$ defined by
\[d_2^{21}(\Omega) \colon \begin{cases}
	t_1^{-1} x_1^2 &\mapsto t_1^{-2}(x_1^2 + x_1y_1 + y_1^2)x_1^2 \\
	t_1^{-1} x_1y_1 &\mapsto t_1^{-2}(x_1^2 + x_1y_1 + y_1^2)x_1y_1 \\
	t_1^{-1} y_1^2 &\mapsto t_1^{-2}(x_1^2 + x_1y_1 + y_1^2)y_1^2 \end{cases}\]
The resulting differential $d_2^{22}(\pic)$ may now be computed by hand. It has been described diagramatically in \cref{d2limitQ8}. We easily see that there's only one possible nonzero element in the kernel, namely $x_1^2 + x_1y_1 + y_1^2$, and hence the kernel is $C_2$ regardless of the field $k$. 

\begin{table}
\setlength{\tabcolsep}{5pt}
\centering
\begin{tabular}{cc|ccccc}
 &  & $t_1^{-2}x_1^4$ & $t_1^{-2}x_1^3 y_1$ & $t_1^{-2}x_1^2y_1^2$ & $t_1^{-2}x_1y_1^3$ & $t_1^{-2}y_1^4$ \\ \hline
$\lambda t_1^{-1}x_1^2$ & $\mapsto$ & $\lambda + \lambda^2$ & $\lambda$ & $\lambda$ & &  \\
$\mu t_1^{-1}x_1y_1$ & $\mapsto$ &  & $\mu$ & $\mu + \mu^2$ & $\mu$ &  \\
$\nu t_1^{-1}y_1^2$ & $\mapsto$ &  &  & $\nu$ & $\nu$ & $\nu + \nu^2$ 
\end{tabular}
\caption{Behaviour of $d_2^{22}(\pic)$ in the limit spectral sequence for $Q_8$. Here $\lambda$ denotes a scalar in $k$. Notice that the differential is not $k$-linear, although it is $\bb{F}_2$-linear.}
\label{d2limitQ8}
\end{table}

\begin{figure}
\centering
\includegraphics{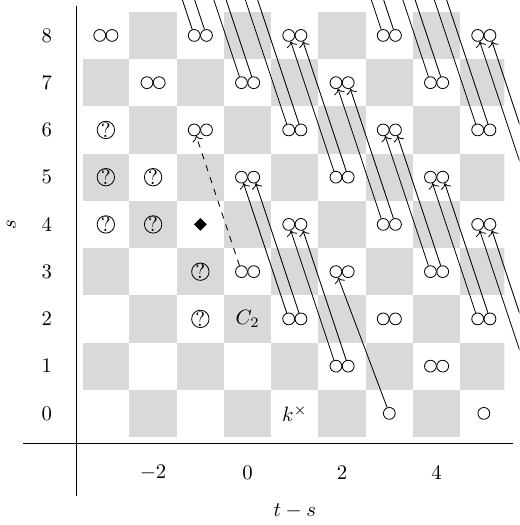}
    \caption{$E_3$-page of the limit spectral sequence for computing the Picard group of $\mi{StMod}\big(kQ_8\big)$. The nontrivial differentials have been illustrated. The group illustrated by the black diamond is the cokernel of the nonlinear map described by \cref{d2limitQ8}.}
\label{E3limitQ8}
\end{figure}

On the $E_3$-page, which has been illustrated in \cref{E3limitQ8}, a similar situation arises: the stable differentials are imported from \cref{E3tateQ8}, but there's a possibly nontrivial unstable differential $d_3^{33}(\pic)$. The corresponding differential $d_3^{32}(\Omega)$ from the Hochschild--Serre spectral sequence is the $k$-linear map 
\[d_3^{32}(\Omega) \colon \begin{cases}
	t_1^{-2}[x_1y_1^2] &\mapsto t_1^{-4}([x_1^3y_1^3] + [x_1^2y_1^4]) \\
	t_1^{-2}[x_1^2y_1] &\mapsto t_1^{-4}([x_1^4y_1^2] + [x_1^3y_1^3]) \end{cases}\]
Using this, we readily compute to find the following result.

\begin{prop}
The behavior of the $d_3^{33}(\pic)$ differential is described by
\[ d_3^{33}(\pic) \colon \begin{cases}
    \lambda t_1^{-2} [x_1 y_1^2] &\mapsto \lambda t_1^{-4} [x_1^2 y_1^4] + \lambda^2 t_1^{-4} [x_1^4 y_1^2] \\
    \mu t_1^{-2} [x_1^2 y_1] &\mapsto \mu^2 t_1^{-4} [x_1^2 y_1^4] + \mu t_1^{-4} [x_1^4 y_1^2]\end{cases}\]
\label{prop:unstableq8diff3}
\end{prop}
Elements in the kernel of this differential correspond to pairs $(\lambda,\mu)$ such that $\lambda + \mu^2 = 0$ and $\lambda^2 + \mu = 0$. Since $k$ is of characteristic~$2$, this corresponds to pairs $(\lambda,\lambda^2)$ such that $\lambda + \lambda^4 = 0$. Clearly, there are trivial solutions $\lambda = 0$ and $\lambda = 1$, but if $k$ contains a primitive cube root of unity $\omega$, then we may also take $\lambda = \omega$ and $\lambda = \omega^2$. We thus find that
\[ \K d_{3}^{33} \cong \begin{cases}
	C_2 \oplus C_2 \qquad &\text{if $k$ has a third root of unity;} \\
	C_2 \qquad &\text{otherwise.}\end{cases}\]

\begin{figure}
\centering
\includegraphics{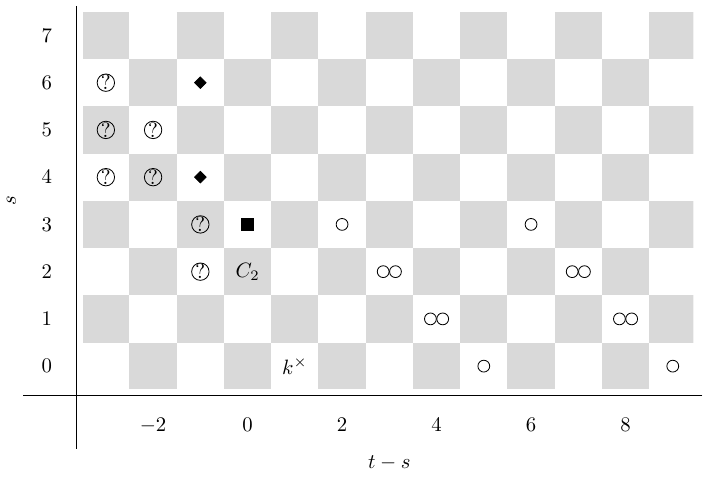}
    \caption{$E_4$-page of the limit spectral sequence for computing the Picard group of $\mi{StMod}\big(kQ_8\big)$. There are no remaining differentials, and the spectral sequence collapses. The black square is $E_3^{33}$, and it depends on the structure of the field $k$. It is either $C_2$ or $(C_2)^2$. The groups illustrated by the black diamonds are the cokernels of the nonlinear maps $d_2^{22}$ and~$d_3^{33}$.}
\label{E4limitQ8}
\end{figure}

We're now ready to write out the $E_4$-page of the limit spectral sequence. A portion of it has been illustrated in \cref{E4limitQ8}. The spectral sequence collapses --- at least in the relevant range $t - s \geq 0$ --- and we find that the line $t - s = 0$ depends on the structure of the field. If $k$ has a third root of unity, then there's a copy of $C_2$ and a copy of $C_2 \oplus C_2$ on the $0$-line, while if $k$ does not have a third root of unity, there's two surviving copies of $C_2$.

In both cases, there's room for nontrivial extension problems. Nonetheless it's easy to overcome these problems: The $4$-fold periodicity of the homotopy groups of $k^{tQ_8}$ implies that the unit is an element of order $4$ in the Picard group. The only groups with the indicated extensions and an element of order $4$ are $C_4$ and $C_4 \oplus C_2$, hence we deduce the following result.

\begin{thm}
\label{resultq8}
	Let $k$ a field of characteristic $2$. Then 
	\[ \Pic\big(\StMod(kQ_{8})\big) \cong \begin{cases}
	    C_4 \oplus C_2 \qquad &\text{if $k$ has a primitive cube root of unity;} \\
	    C_4 \qquad &\text{otherwise.}\end{cases}\]
\end{thm}

\begin{rmk}
    The exotic generator of the Picard group of $\StMod(kQ_8)$ has a known explicit description as a $G$-representation. It is captured by the associations
    \[ i \mapsto \begin{pmatrix}
        1 & 0 & 0 \\ 1 & 1 & 0 \\ 0 & 1 & 1 \end{pmatrix} \qquad \text{and} \qquad j \mapsto \begin{pmatrix} 1 & 0 & 0 \\ \omega & 1 & 0 \\ 0 & \omega^2 & 1 \end{pmatrix}\text{,}\]
    where $\omega$ denotes a principal cube root of unity. It would be interesting to have a homotopical construction of this object.
\end{rmk}

\begin{rmk}
\label{rmk:q8}
    This method of computing the group of endotrivial modules differs dramatically from the work of Carlson--Th\'evenaz \cite{Carlson2000-ux}, who used representation-theoretic techniques (namely, the theory of support varieties). In the case of the quaternion group $Q_8$, they explicitly construct the endotrivial modules above, and prove that no other endotrivial modules exist. 
    
    However, our method allows us to compute the group of endotrivial modules \textit{a priori}.  Indeed, our method is non-constructive, and moreover it provides a conceptual explanation for why the cube root of unity $\omega$ is significant in the case $G \cong Q_8$: it arises due to the existence of a nonlinear $d_3$ differential in the Picard spectral sequence (\cref{prop:unstableq8diff3}), whose kernel depends on the existence of a cube root of unity.
\end{rmk}

\subsection{The case of generalized quaternion groups}

We shall now calculate the Picard spectral sequence
\[ E_{2}^{st} = H^s\big(D_{2^{n-1}}; \pi_t \mf{pic} \,\mi{StMod}(kC_2)\big) \Ra \pi_{t - s} \mf{pic} \,\mi{StMod}(kQ_{2^{n}})\,\text{.}\]
The $E_2$-page, illustrated on \cref{E2limitQ2n}, looks effectively the same as that of $Q_8$, and indeed is computed in the same way. The differentials in the stable range are imported from the associated HFPSS
\[\begin{split}
E_2^{st} &\cong H^s(D_{2^{n-1}};k) \otimes \pi_t(k^{tC_2}) \\&\cong k[x_1,u_1,z_2] / (u_1 x_1 + x_1^2) \otimes k[t^{\pm 1}] \Ra \pi_*(k^{tQ_{2^n}})
\end{split}\]
which we computed in \cref{generalizedquaternionfaithful}. We found that $d_2(t_1) = u_1^2 + z_2$ and $d_3(t_1^2) = u_1z_2$, and that the pages looked identical to the analogous spectral sequences for $Q_8$, which were illustrated in \cref{tq-E2hfpss} and \cref{E3tateQ8}.

\begin{figure}
\centering
\includegraphics{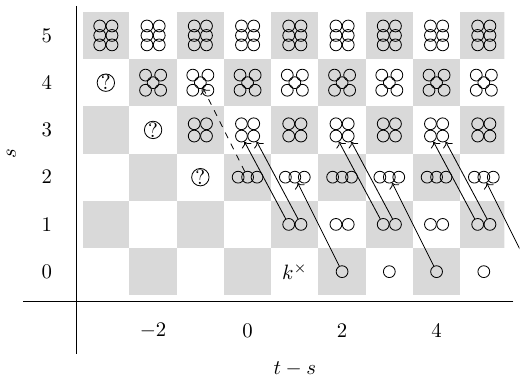}
    \caption{$E_2$-page of the limit spectral sequence for computing the Picard group of $\mi{StMod}\big(kQ_{2^n}\big)$. The nontrivial differential has been illustrated. Terms indicated by a question mark are the higher cohomology of $D_{2^{n-1}}$ with coefficients in $k^{\times}$, and are likely well-behaved when $k$ is perfect. The dashed diagonal is governed by \cref{unstabledifferential}.}
\label{E2limitQ2n}
\end{figure}




The crucial unstable differential is again $d_2^{22}(\pic)$, which we compute through \cref{unstabledifferential}. The differential $d_2^{21}(\Omega)$ is the $k$-linear map $k^3 \to k^5$ defined by
\[d_2^{21}(\Omega) \colon \begin{cases}
	t^{-1}u^2 &\mapsto t^{-2}(u^2+z)u^2 \\
	t^{-1}z &\mapsto t^{-2}(u^2+z)z \\
	t^{-1}ux_1 &\mapsto t^{-2}(u^2+z)ux_1\end{cases}\]
We use this to compute $d_2^{22}$ by hand; the result has been indicated in \cref{d2limitQ2n}. The only nonzero element in the kernel is $t_1^{-1}(u_1^2 + z_2)$, so the kernel is $C_2$ regardless of the field $k$. This brings us to the $E_3$-page, illustrated in \cref{E3limitQ2n}. 

\begin{table}
\setlength{\tabcolsep}{5pt}
\centering
\begin{tabular}{cc|ccccc}
 &  & $t_1^{-2}u^4$ & $t_1^{-2}u_1^2z_2$ & $t_1^{-2}z_2^2$ & $t_1^{-2}u_1^3x_1$ & $t_1^{-2}u_1z_2x_1$ \\ \hline
$\lambda t_1^{-1}u_1^2$ & $\mapsto$ & $\lambda + \lambda^2$ & $\lambda$ & & &  \\
$\lambda t_1^{-1}z_2$ & $\mapsto$ &  & $\lambda$ & $\lambda + \lambda^2$ & &  \\
$\lambda t_1^{-1}u_1x_1$ & $\mapsto$ &  &  &  & $\lambda + \lambda^2$  & $\lambda$
\end{tabular}
\caption{Behaviour of $d_2^{22}(\pic)$ in the limit spectral sequence for $Q_{2^n}$. Here $\lambda$ denotes a scalar in $k$. Compare with \cref{d2limitQ8}.}
\label{d2limitQ2n}
\end{table}


\begin{figure}
\centering
\includegraphics{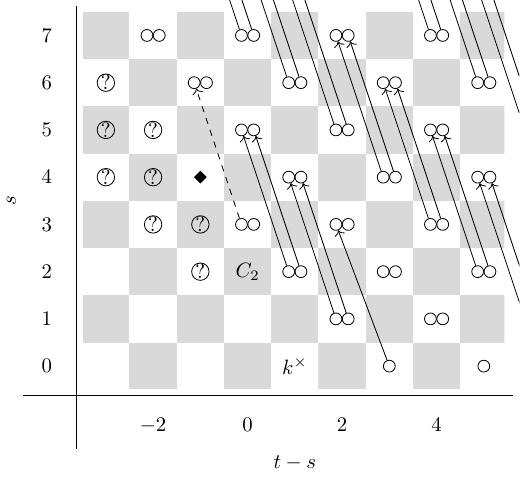}
    \caption{$E_3$-page of the limit spectral sequence for computing the Picard group of $\mi{StMod}\big(kQ_{2^n}\big)$. The nontrivial differentials have been illustrated. The group illustrated by the black diamond is the cokernel of the nonlinear map described by \cref{d2limitQ2n}.}
\label{E3limitQ2n}
\end{figure}

On the $E_3$-page there's again an unstable differential, $d_3^{33}(\pic)$. In the analaysis of the HFPSS, the differential $d_3^{32}(\Omega)$ was determined to be the map $k^2 \to k^2$ defined by
\[d_3^{32}(\Omega) \colon \begin{cases}
	t_1^{-2}[u_1z_2] &\mapsto t_1^{-4}[u_1^2z_2^2] \\
	t_1^{-2}[z_2x_1] &\mapsto t_1^{-4}[u_1z_2^2x_1] \end{cases}\]
Using \cref{unstabledifferential} again, we compute $d_3^{33}(\pic)$ by hand again to find that 
\[d_3^{33}(\pic) \colon \begin{cases}
    \lambda t_1^{-2} [u_1 z_2] &\mapsto \lambda t_1^{-4} [u_1^2 z_2^2] + \lambda t_1^{-4} [u_1 z_2^2 x_1] \\
    \mu t_1^{-2} [z_2 x_1] &\mapsto \mu^2 t_1^{-4} [u_1^2 z_2^2] + \mu^2 t_1^{-4} [u_1 z_2^2 x_1] \end{cases}\]
We see that for an element to be in the kernel of $d_3^{33}$, we need $\lambda + \lambda^2$ and $\mu + \mu^2$ to be $0$. Therefore, both $\lambda$ and $\mu$ can be either $0$ or $1$, so that the kernel is isomorphic to $(C_2)^2$, irrespective of the field $k$. 




\begin{figure}
\centering
\includegraphics{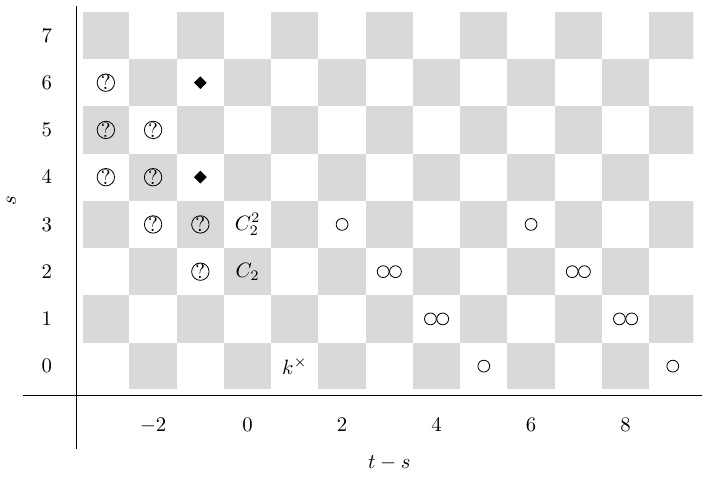}
    \caption{$E_4$-page of the limit spectral sequence for computing the Picard group of $\mi{StMod}\big(kQ_{2^n}\big)$. There are no remaining differentials, and the spectral sequence collapses. The groups illustrated by the black diamonds are the cokernels of the nonlinear maps $d_2^{22}$ and~$d_3^{33}$.}
\label{E4limitQ2n}
\end{figure}

The relevant portion of $E_4$-page is now in \cref{E4limitQ2n}. There are no further differentials which may influence the line $t - s = 0$. As in the case of $Q_8$, there's room for a nontrivial extension problem, which is resolved by observing the $4$-fold periodicity of the Tate cohomology of $Q_{2^n}$. We thus conclude the following result.

\begin{thm}
	\label{resultQ2n}
	The Picard group of $\mi{StMod}\big(kQ_{2^n}\big)$, where $n \geq 4$, is given by $C_2 \oplus C_4$ for all fields $k$ of characteristic~$2$.
\end{thm}

\begin{rmk}
\label{rmk:q2n}    
    To compute the group of endotrivial modules of the generalized quaternion groups, Carlson--Th{\'e}venaz rely on the computation for $Q_8$. More precisely, they prove the following theorem:
    \begin{thm*}[{\cite[Thm. 2.7]{Carlson2000-ux}}]
        Let $G$ be a non-cyclic $p$-group, and let $\mathcal{E}$ denote the family of subgroups $H$ such that $H$ is an extraspecial $2$-group that is not isomorphic to $D_8$, or an almost extraspecial $2$-group, or an elementary abelian group of rank 2. Then the restriction map
        \[\textnormal{Res} \colon \Pic\big(\mathsf{StMod}(kG)\big) \to \prod_{H \in \mathcal{E}} \Pic\big(\mathsf{StMod}(kH)\big)\]
         is injective.
    \end{thm*}
    They then apply it to $Q_{2^n}$: Noting that $Q_8$ naturally sits in $Q_{2^n}$ as a subgroup, they study this restriction map to explicitly construct the endotrivial modules for $Q_{2^n}$. In contrast, with our method, the computations for the generalized quaternion groups are logically independent of the computations for $Q_8$. 
\end{rmk}

\appendix

\section{Appendix: Products in Tate Cohomology}
\label{appendixtate}

In this section, we describe the multiplicative structure of $\pi_{-*}(k^{tG}) = \widehat{H}^*(G;k)$ for certain classes of $p$-groups $G$ (namely, for $G$ Cohen--Macaulay).  These ideas are certainly not new or original, but are necessary for our Tate spectral sequence calculations.

First recall that for $G$ a cyclic $p$-group, the Tate cohomology rings can be obtained by inverting the polynomial generator in the group cohomology ring $H^*(G;k)$. One can show this by computing explicit resolutions, or via spectral sequence computations as we did in \cref{sectiongaloisdescent}.

\begin{lem}
	\label{tatereference}
	The ring structure on Tate cohomology $\wh{H}^*(C_{p^n};k)$ is described as follows.
	\begin{align*}
		\wh{H}^*(C_{p^n};k) &\cong k[e^{\pm 1}] & & \text{if $p^n = 2$;} \\
		\wh{H}^{*}(C_{p^n};k) & \cong k[e_1,e_2^{\pm 1}] / (e_1^2) \qquad & &\text{if $p^n \geq 3$.} \end{align*}
	Here $|e| = 1$, $|e_1| = 1$, and $|e_2| = 2$, and they may be identified with the generators on group cohomology. 
\end{lem}

For elementary abelian groups $G \cong (C_p)^n$ with $n \geq 2$, the ring structure can also be derived from the product in group cohomology $H^*((C_p)^n;k)$, but is more complicated.  For example, $\wh{H}^*((C_p)^n;k)$ is not finitely generated as a ring or a module over $H^*((C_p)^n;k)$.  This follows from work of Benson-Carlson, who in fact prove that for a group of $p$-rank at least 2, there are no non-zero products in negative Tate cohomology:

\begin{thm}[{\cite[Theorem 3.1]{Benson1992-df}}]
\label{thm:prodnegcohomology}
Suppose that the $p$-rank of $G$ is two or more.  If $H^*(G;k)$ is Cohen-Macaulay, then 
\[\wh{H}^m(G;k) \cdot \wh{H}^n(G;k) = 0\]
for all $m,n < 0$.
\end{thm}

\begin{rmk}
Note that the groups $G$ of $p$-rank 1 are either cyclic or (generalized) quaternion.  We saw in  \cref{tatereference} that the Tate cohomology of cyclic groups was periodic.  For (generalized) quaternion groups $Q$, the Tate cohomology $\wh{H}^*(Q;k)$ is also periodic:
\begin{align*}
		\wh{H}^*(Q_8;k) &\cong k[x_1, x_2, e_4^{\pm1}]/(x_1^2 + x_1x_2 + x_2^2, x_1^2x_2 + x_1x_2^2) \\
		\wh{H}^*(Q_{2^n};k) &\cong k[x_1, x_2, e_4^{\pm1}]/(x_1x_2, x_1^3 + x_2^3)
\end{align*}

\end{rmk}


For our endo-trivial module calculations, we are especially interested in the ring structure of $\pi_{*}(k^{tG})$ in the case that $G$ is elementary abelian or a dihedral $p$-group. These groups belong to a class of groups called the \textbf{Cohen--Macaulay groups}, whose Tate cohomology rings can be understood especially well. In fact, we will explicitly describe the ring structure on $\wh{H}^*(G;k)$ using the Greenlees spectral sequence, thereby recording a proof of \cref{thm:prodnegcohomology} along the way.

\subsection{Cohen--Macaulay groups}

We begin by recalling the definitions. Recall that if $R$ is a commutative Noetherian local ring, then we have the notions of depth and Krull dimension.

Let $\mathfrak{m}$ be the maximal ideal of a commutative Noetherian local ring $R$.  The \textbf{depth} of $R$ is the smallest integer $i$ such that $\Ext^i_R(R/\mathfrak{m},R) \neq 0$.  Equivalently, the depth of $R$ is the supremum of the lengths of regular sequences in $\mathfrak{m}$.  The \textbf{Krull dimension} of $R$ is the supremum of the lengths of chains of prime ideals in $R$. In general, $\textnormal{depth}(R) \leq \textnormal{dim}(R)$.  If equality holds, we say that $R$ is a \textbf{Cohen--Macaulay} ring. 

We are interested in the case $R = H^*(G;k)$. These rings are finitely generated, graded-commutative $k$-algebras over $R_0 = k$. In this case, one can use Noether normalization to construct a graded polynomial subring $A = k[x_1, \ldots, x_n]$ such that $R = H^*(G;k)$ is finitely generated over $A$.  In this situation, we can easily identify Cohen-Macaulay rings using Hironaka's criterion (also known as miracle flatness).

\begin{lem}[``Miracle Flatness'']
	\label{miracleflatness}
	Let $R$ be a finitely generated, graded, commutative $k$-algebras over $R_0 = k$, with $R_i = 0$ for $i < 0$, that is finitely generated as a module over a graded polynomial subring $A = k[x_1, \ldots, x_n]$. Then $R$ is Cohen--Macaulay if and only if $R$ is free over~$A$.
\end{lem}

Using this result, direct computations of the cohomology rings make it clear that $(C_p)^n$, $D_{2^n}$, and $Q_{2^n}$ are Cohen--Macaulay:
\begin{align*}
        H^*((C_p)^n;k) &\cong \left\{
		\begin{array}{ll}
		k[x_1, \ldots, x_n]  & |x_i| = 1, \ p = 2 \\
		k[x_1, \ldots, x_n] \otimes \Lambda(y_1, \ldots, y_n) & |x_i| = 2, |y_i| = 1,  \ p \neq 2
		\end{array}
		\right.\\
        H^*(D_{2^{n-1}}; k) &\cong k[x_1,u, z]/(ux_1 +x_1^2)\\
		H^*(Q_8;k) &\cong k[x_1, x_2, e_4]/(x_1^2 + x_1x_2 + x_2^2, x_1^2x_2 + x_1x_2^2) \\
		H^*(Q_{2^n};k) &\cong k[x_1, x_2, e_4]/(x_1x_2, x_1^3 + x_2^3)
\end{align*}

In general, the depth and dimension of $H^*(G;k)$ is closely related to that of the elementary abelian subgroups $(C_p)^n \subset G$.  For example, Quillen \cite{Quillen71} proved that the dimension of $H^*(G; k)$ is equal to the $p$-rank of the largest elementary abelian subgroup of $G$.  That is, the largest integer $n$ such that $(C_p)^n$ is a subgroup of $G$. Furthermore, a theorem of Duflot \cite{Duflot1981} tells us that the depth of $H^*(G; k)$ is greater than or equal to the $p$-rank of the largest \textit{central} elementary abelian subgroup of $G$.  That is, the largest integer $n$ such that $(C_p)^n$ is a \textit{central} subgroup of $G$.  The idea is that one can find a homogeneous system of parameters in $H^*(G;k)$ that restricts to a regular sequence in $H^*((C_p)^n;k)$.   One then uses induction to show this homogeneous system of parameters is in fact a regular sequence.

As a consequence, if $H^*(G; k)$ is Cohen--Macaulay, then the depth is equal to the dimension is equal to the $p$-rank.  For $H^*((C_p)^n;k)$, the $p$-rank is $n$, and for $H^*(Q_{2^n};k)$, the $p$-rank is 1. Note that for $H^*(D_{2^{n-1}};k)$, the $p$-rank is 2, whereas the central $p$-rank is 1.

\subsection{The {\v C}ech cohomology spectral sequence}

We now study multiplication in Tate cohomology of $p$-groups $G$ by studying the ring spectrum~$k^{tG}$.  In particular, we interpret $k^{tG}$ as a \v{C}ech spectrum $R[I^{-1}]$, which allows us to do computations using the \v{C}ech cohomology spectral sequence.
We first discuss the Koszul and \v{C}ech constructions of Greenlees--May, cf. \cite{may1995completions}. These generalize the notion of the flat stable Koszul chain complex and the \v{C}ech complex in commutative algebra.

Let $R$ be an $\bb{E}_\infty$-ring spectrum. Given $x \in \pi_*(R)$, we define the Koszul spectrum $K(x)$ as the fibre occurring in the fibre sequence
\[ K(x) \to R \to R[x^{-1}] \text{.} \]
More generally, $I = (x_1, \ldots, x_n)$ is a finitely generated ideal in $\pi_* (R)$, then we define the \textbf{Koszul spectrum} $K(I)$ as a tensor product $K(x_1) \otimes_R \cdots \otimes_R K(x_n)$. Up to homotopy, this construction depends only on the radical of the ideal $I$. 

We now define the \textbf{{\v C}ech spectrum} $R[I^{-1}]$ to be the cofiber of the map $K(I) \to R$. We may regard it as the localisation away from the ideal $I$.  Note that if $I = (x)$ is principal, then the \v{C}ech spectrum $R[I^{-1}]$ is precisely $R[x^{-1}]$.  However, for an arbitrary finitely generated ideal $I = (x_1, \ldots, x_n)$, $R[I^{-1}]$ is generally not the same as the localization at a multiplicatively closed subset of $\pi_*(R)$, cf. Theorem 5.1 of \cite{may1995completions}. Nevertheless, there exists a \v{C}ech cohomology spectral sequence
\[E_2 \cong \check{C}H^{-s,-t}_{I}\big(\pi_*(R)\big) \Rightarrow \pi_{s+t}\big(R[I^{-1}]\big)\,\text{,}\]
which allows us to compute the homotopy groups of $R[I^{-1}]$ using methods of commutative algebra.

As it happens, one may identify the Tate fixed points $k^{tG}$ with a {\v C}ech spectrum, thanks to the following theorem of Greenlees.

\begin{thm}[{\cite[Thm. 4.1]{Greenlees95commutativealgebra}}]
\label{thm:greenleestate}
    Let $G$ be a $p$-group acting trivially on the Eilenberg--MacLane spectrum $k$.  Let $R = k^{hG}$ be the homotopy fixed points, so that $\pi_{-*}(R) \cong H^*(G;k)$. Define $I$ to be the augmentation ideal $I = \text{ker}(H^*(G;k) \to k)$. Then there is a homotopy equivalence between $R[I^{-1}]$ and $k^{tG}$.
\end{thm}

If $H^*(G; k)$ is Cohen--Macaulay, then by miracle flatness it is free over a polynomial subalgebra $A_* \subseteq H^*(G;k)$, say  $A_* = k[\zeta_1, \ldots, \zeta_n]$. Note that the radical of the ideal $J = (\zeta_1, \ldots, \zeta_n)$ is the ideal $I$ of elements in positive degrees.  Hence $k^{tG} \simeq R[I^{-1}] \simeq R[J^{-1}]$. Moreover, since $R_* = H^*(G;k)$ is free over $A_*$, we can reduce from understanding $R[I^{-1}]$ to understanding $A[J^{-1}]$. Hence we study the \v{C}ech cohomology spectral sequence:
\[E_2 \cong \check{C}H^*_{(\zeta_1, \ldots, \zeta_n)}\big(k[\zeta_1, \ldots, \zeta_n]\big) \Rightarrow \pi_*(A[J^{-1}])\,\text{.}\]
It is easy to calculate the $E_2$ page of this spectral sequence, as we have induced long exact sequences relating \v{C}ech cohomology to local cohomology coming from the fiber sequence $K(I) \to R \to R[I^{-1}]$:

\begin{thm}[\cite{may1995completions}]
    For an $R$-module $M$, we have an exact sequence
    \[\begin{tikzcd}
    	0 \ar[r] & H_0^I(M) \ar[r] & M \ar[r] & \check{C}H^0_I(M) \ar[r] & H_1^I(M) \ar[r] & 0 \end{tikzcd}\]
    and an isomorphism
    \[ H^I_s(M) \cong \check{C}H^{s-1}_I(M) \qquad \text{for all $s \geq 1$.}\]
\end{thm}

Since the $(\zeta_1, \ldots, \zeta_n)$ form a regular sequence in $A_* = k[\zeta_1, \ldots, \zeta_n]$, we observe that $H^j_J(A_*) = 0$ for $j \neq \text{dim}(A_*)=n$.  Hence the $E_2$-page of this spectral sequence is concentrated in two rows, at $s=0$ and $s=n-1$, where we have $\check{C}H_J^0(A_*) \cong A_*$, and $\check{C}H_J^{n - 1}(A_*) \cong k[x_1^{-1},\ldots,x_n^{-1}]$.

The multiplication in the spectral sequnce recovers the multiplication structure on $A[J^{-1}]$.  This allows us to compute $\pi_*(R[I^{-1}]) \cong \pi_*(k^{tG}) \cong \widehat{H}^*(G;k)$ for a Cohen--Macaulay group $G$; indeed, \cref{miracleflatness} tells us that $\pi_*(k^{tG})$ is free over $\pi_*(A[J^{-1}])$ so we may tensor up the spectral sequence to $\pi_*(k^{tG})$ without exactness issues. Note that these generators live only in positive degrees in cohomology, and are represented in Adams degree $(-t_i,0)$.  That is, they lower $t-s$ degree, and preserve $s$ degree.

\subsection{Tate cohomology of elementary abelian groups}
\label{appendix:elemab}

Let us now turn to an explicit description for elementary abelian groups. Recall that if $k$ is of characteristic $p$, and $G \cong (C_p)^n$, then
\[H^*(G;k) = \left\{
		\begin{array}{ll}
		k[x_1, \ldots, x_n]  & |x_i| = 1, \ p = 2 \\
		k[x_1, \ldots, x_n] \otimes \Lambda(y_1, \ldots, y_n) & |x_i| = 2, |y_i| = 1,  \ p \neq 2
		\end{array}
		\right.\]
We have a regular sequence $I = (x_1, \ldots, x_n)$, and moreover we know that $A_* = k[x_1, \ldots, x_n]$ corresponds to $\pi_*(k^{h\mathbb{T}^n}) \cong k[x_1, \ldots, x_n]$ where $|x_i| = -1$ for $p=2$, and $|x_i| = -2$ for $p$ odd. To compute $\pi_*(k^{t(C_p)^n})$, we first calculate the \v{C}ech cohomology spectral sequence
\[E_2 \cong \check{C}H^*_{I}\big(\pi_*(k^{h\mathbb{T}^n})\big) \Rightarrow \pi_*(k^{t\mathbb{T}^n})\,\text{.}\]
As before, the $E_2$ page of this spectral sequence is concentrated in two rows, at $s=0$ and $s=n-1$. In these rows, we have $\check{C}H^0_{I}\big(\pi_*(k^{h\mathbb{T}^n})\big) \cong \pi_*\big(k^{h\mathbb{T}^n}\big)$ and $\check{C}H^{n-1}_{I}\big(\pi_*(k^{h\mathbb{T}^n})\big) \cong k[x_1^{-1}, \ldots, x_n^{-1}]$, shifted in $t$-degree by $n$ for $p=2$ (Figure \ref{CechP2}), or $2n$ for $p$ odd (Figure \ref{CechPOdd}).

\begin{figure}
\centering
\begin{tikzpicture}
	\matrix(m)[matrix of math nodes, nodes in empty cells, nodes = {minimum width = 8ex, minimum height = 4ex, outer sep = 1ex}, column sep = 1ex, row sep = 1ex]
	{              &   &     &     &     & \\
	n-1      &  &  &  &  & k & k^{d_1} & k^{d_2} & k^{d_3} & \\
	\vdots      & &  &  &  &  &  &  &  & \\
	0      & k^{d_3} & k^{d_2} & k^{d_1} & k &  & &  &  & \\
	\strut & -3 & -2 & -1 & 0  &  1 & 2 & 3 & 4 & \strut \\};
	
	\draw[thick] (m-1-1.south east) -- (m-5-1.east);
 	\draw[thick] (m-5-1.north) -- (m-5-10.north west);
	
\end{tikzpicture}
\caption{The Adams graded $E_2 = E_\infty$ page of the \v{C}ech cohomology spectral sequence computing $\pi_*(k^{t\T^n})$, with $p=2$.  Let $d_i = \textnormal{dim}(\pi_{i}(k^{h\T^n})) =  {n-1 + i \choose n-1}$.}
\label{CechP2}
\end{figure}

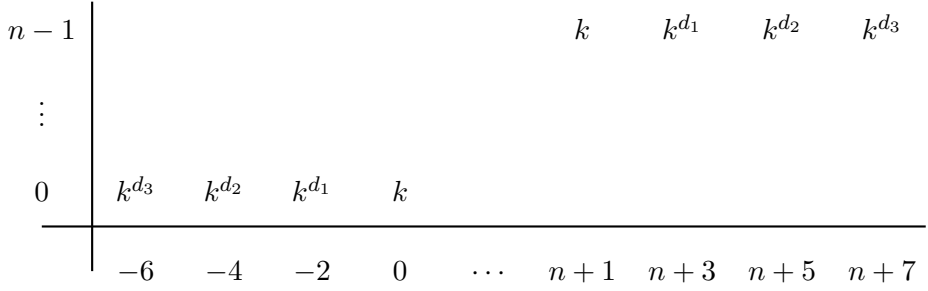
\begin{figure}
\centering
\begin{tikzpicture}
	\matrix(m)[matrix of math nodes, nodes in empty cells, nodes = {minimum width = 6ex, minimum height = 4ex, outer sep = 1ex}, column sep = 1ex, row sep = 1ex]
	{              &   &     &     &     & \\
	n-1      &  &  &  &  & &  k & k^{d_1} & k^{d_2} & k^{d_3} & \\
	\vdots      & &  &  &  &  & &  &  &  & \\
	0      & k^{d_3} & k^{d_2} & k^{d_1} & k & &  & &  &  & \\
	\strut & -6 & -4 & -2 & 0 & \cdots &  n+1 & n+3 & n+5 & n+7 & \strut \\};
	
	\draw[thick] (m-1-1.south east) -- (m-5-1.east);
 	\draw[thick] (m-5-1.north) -- (m-5-11.north west);
	
\end{tikzpicture}
\caption{The Adams graded $E_2 =E_\infty$ page of the \v{C}ech cohomology spectral sequence computing $\pi_*(k^{t\T^n})$, with $p$ odd.  Let $d_i = \dim(\pi_{2i}(k^{h\T^n})) =  {{n-1 + i} \choose {n-1}}$.}
\label{CechPOdd}
\end{figure}

The spectral sequence collapses for degree reasons, and note also that there is nowhere for multiplication by two elements in positive degree to land, and so therefore must vanish.  The multiplicative structure on the $E_\infty$-page gives us the multiplicative structure on $\pi_*(k^{t\mathbb{T}^n})$, from which we infer the following result.  

\begin{thm}\label{thm:tate-mult}
The multiplication in $\pi_*(k^{t\mathbb{T}^n})$ is described in the following way: In negative degrees, multiplication in $ \pi_*(k^{t\mathbb{T}^n})$ is the same as multiplication in $\pi_*(k^{h\mathbb{T}^n})$. In positive degrees, we have a class $\alpha \in \pi_{1}(k^{t\mathbb{T}^n})$ represented in Adams degree $(1,n-1)$ for $p=2$, or $\alpha \in \pi_{n+1}(k^{t\mathbb{T}^n})$ represented in Adams degree $(n+1,n-1)$ for $p$ odd.  This class $\alpha$ is a generator for the algebra $\check{C}H_J^{n - 1}(k^{h\mathbb{T}^n}) \cong k[x_1^{-1}, \ldots, x_n^{-1}]$.  One has $\alpha \cup x_i = 0$ and $\alpha \cup \alpha= 0$.  This, along with the $\pi_*(k^{h\mathbb{T}^n})$-algebra structure of $k[x_1^{-1}, \ldots, x_n^{-1}]$, determines the multiplication for all the terms in positive degrees. 
\end{thm}

From this, we infer what the multiplication in $\pi_{-*}\big(k^{t(C_p)^n}\big) \cong \widehat{H}^*((C_p)^n;k)$ must be. For $p=2$, since the group cohomology $\pi_{-*}\big(k^{h(C_p)^n}\big) \cong H^*((C_p)^n;k)$ is precisely the polynomial algebra $\pi_{-*}(k^{h\mathbb{T}^n})$, the multiplication in $\widehat{H}^*((C_p)^n;k)$ is precisely the same as the multiplication in $\pi_{-*}(k^{t\mathbb{T}^n})$.  In particular, we provide the computation for the Tate cohomology of $(C_2)^2$ (Figure \ref{tc22}).

For $p$ odd, since the group cohomology $\pi_{-*}\big(k^{h(C_p)^n}\big) \cong H^*((C_p)^n;k)$ is free over the polynomial subalgebra $\pi_{-*}(k^{h\mathbb{T}^n})$, we simply need to tensor the \v{C}ech cohomology spectral sequence computing $\pi_*(k^{t\mathbb{T}^n})$ with the exterior algebra $\Lambda (y_1, \ldots, y_n)$ where $y_i$ is represented in Adams degree $(-1,0)$ to obtain the multiplication in $\widehat{H}^*((C_p)^n;k)$. 

\begin{figure}
\centering
\includegraphics{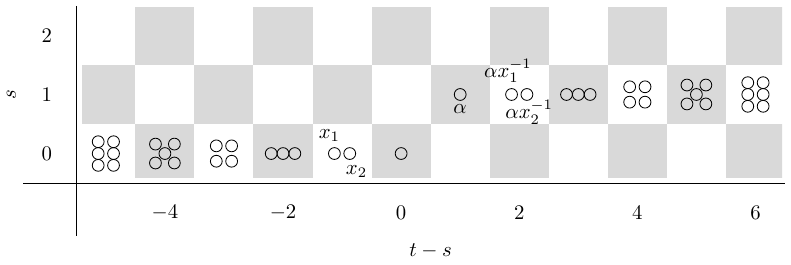}
    \caption{The Adams-graded $E_2$-page of the {\v C}ech cohomology spectral sequence computing $\pi_* \big(k^{t(C_2)^2}\big)$. There are no differentials and the spectral sequence collapses.}
\label{tc22}
\end{figure}

\subsection{Tate cohomology of dihedral $p$-groups}
\label{appendix:dihedral}

Recall that for $n \geq 3$, $H^*(D_{2^n};k) \cong k[x_1, x_2, z]/(x_1x_2)$, where $|x_i| = 1$ and $|z| = 2$. Moreover, $\textnormal{Sq}^1(z) = (x_1+x_2)z$.  It is convenient to set $u = x_1 + x_2$ and write
\[H^*(D_{2^{n-1}}; k) \cong k[x_1,u, z]/(ux_1 +x_1^2)\text{.}\]
This calculation is standard in the literature, by using induction and the cohomological Serre spectral sequence associated to the fibration $BC_2 \to BD_{2^n} \to BD_{2^{n-1}}$.  One can restrict to certain subgroups of $D_{2^{n-1}}$ to determine the differentials.

Note that $H^*(D_{2^n};k)$ is also Cohen--Macaulay with ideal $I = (u,z)$. Furthermore, we have the quotient $H^*(D_{2^n};k) / (u,z) \cong \Lambda(x_1)$.  We can take $A$ such that $A_* \cong k[u,z]$, and we therefore obtain the \v{C}ech cohomology spectral sequence (Figure \ref{cechdihedral}).

\begin{figure}
\centering
\includegraphics{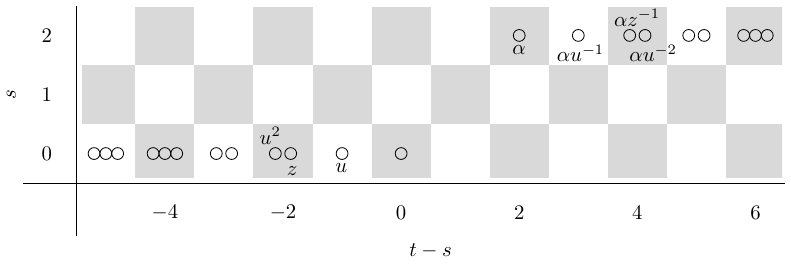}
    \caption{The Adams-graded $E_2$-page of the {\v C}ech cohomology spectral sequence computing $\pi_* \big(A[I^{-1}]\big)$. There are no differentials and the spectral sequence collapses.}
\label{cechdihedral}
\end{figure}

The multiplication in $\pi_*(A[I^{-1}])$ is described in the following way: In negative degrees, multiplication in $ \pi_*(A[I^{-1}])$ is the same as multiplication in $A_* \cong k[u,z]$. In positive degrees, we have a class $\alpha \in (2,0)$, which generates the algebra $\check{C}H_J^{n - 1}(A_*) \cong k[u^{-1}, z^{-1}]$.  One has $\alpha \cup u = \alpha \cup z = 0$ and $\alpha \cup \alpha= 0$.  This, along with the $A_*$-algebra structure of $k[u^{-1},z^{-1}]$, determines the multiplication for all the terms in positive degrees. 

From this, we infer what the multiplication in $\wh{H}^*(D_{2^n};k)$ must be by tensoring the spectral sequence (Figure \ref{cechdihedral}) with the exterior algebra $\Lambda(x_1)$, where $x_1$ is represented in Adams degree $(-1,0)$ (Figure \ref{tatedihedral}).

\begin{figure}
\centering
\includegraphics{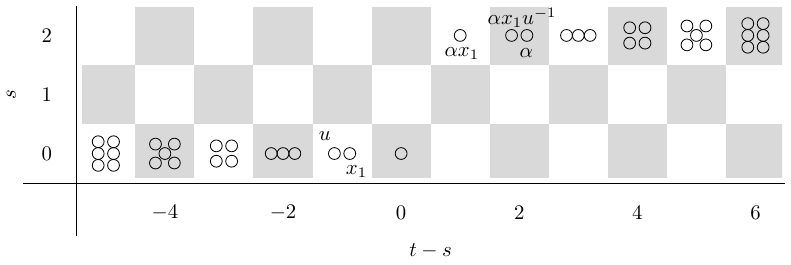}
    \caption{The Adams-graded $E_2$-page of the \v{C}ech cohomology spectral sequence computing $\pi_*\big(k^{tD_{2^n}}\big)$. There are no differentials again.}
\label{tatedihedral}
\end{figure}

\pagebreak

\section*{Acknowledgements}
\label{acknowledgements}
\addcontentsline{toc}{section}{\nameref{acknowledgements}}

The first author is grateful to Jesper Grodal and Kaif Hilman for their comments and insights, and to Alexis Aumonier for proofreading part of the document. The second author would especially like to thank Akhil Mathew and Vesna Stojanoska for catching errors and helpful discussions about this project.

\vspace{3mm}\noindent
\framebox[\textwidth]{
\begin{tabular*}{0.96\textwidth}{@{\extracolsep{\fill} }cp{0.84\textwidth}}
\raisebox{-0.7\height}{
\begin{tikzpicture}[y = 0.80pt, x = 0.8pt, yscale = -1, inner sep = 0pt, outer sep = 0pt, scale = 0.12]
    \definecolor{c003399}{RGB}{0,51,153}
    \definecolor{cffcc00}{RGB}{255,204,0}
    \begin{scope}[shift={(0,-872.36218)}]
        \path[shift={(0,872.36218)},fill=c003399,nonzero rule] (0.0000,0.0000) rectangle (270.0000,180.0000);
        \foreach \myshift in 
            {(0,812.36218), (0,932.36218), 
            (60.0,872.36218), (-60.0,872.36218), 
            (30.0,820.36218), (-30.0,820.36218),
            (30.0,924.36218), (-30.0,924.36218),
            (-52.0,842.36218), (52.0,842.36218), 
            (52.0,902.36218), (-52.0,902.36218)}
        \path[shift=\myshift,fill=cffcc00,nonzero rule] (135.0000,80.0000) -- (137.2453,86.9096) -- (144.5106,86.9098) -- (138.6330,91.1804) -- (140.8778,98.0902) -- (135.0000,93.8200) -- (129.1222,98.0902) -- (131.3670,91.1804) -- (125.4894,86.9098) -- (132.7547,86.9096) -- cycle;
    \end{scope}
\end{tikzpicture}}
&
\footnotesize{The first author was supported by the European Research Council (ERC) under the European Union's Horizon 2020 research and innovation programme (grant agreement N\textsuperscript{\underline{o}} 682922).}
\end{tabular*}
}

\clearpage
 
\addcontentsline{toc}{section}{References}
\setlength{\parskip}{0.2cm}
\bibliographystyle{alpha}
\bibliography{PicGaloisDescent}

\end{document}